\documentclass[12pt]{article}

\usepackage{amsmath,hyperref}
\hypersetup{
    colorlinks=true,
    allcolors={blue}
}
\usepackage{amsfonts}
\usepackage{mathrsfs }
\usepackage{amssymb}
\usepackage{amsthm}
\usepackage{mathtools}
\usepackage{graphicx}
\usepackage{setspace}
\usepackage{enumerate}
\usepackage{tikz}
\usepackage{url,color}
\usepackage[paperheight=11in,paperwidth=8.5in,margin=0.75in]{geometry}
\usepackage{chngcntr}
\usepackage{multicol}
\everymath{\displaystyle}
\newtheorem{thm}{Theorem}[section]

\newtheorem{cor}[thm]{Corollary}
\newtheorem{prop}[thm]{Proposition}

\newtheorem{lemma}[thm]{Lemma}
\newcommand{\fS}{{\mathfrak S}}

\DeclareMathOperator{\NCSym}{NCSym}

\DeclareMathOperator{\sign}{sign}
\DeclareMathOperator{\wt}{wt}

\definecolor{vividviolet}{rgb}{0.62, 0.0, 1.0}
\definecolor{OliveGreen}{rgb}{0.0, .6, .1}

\newcommand{\cA}{{\cal A}}

\newcounter{bar}
\newcommand{\case}{%
	\stepcounter{bar}%
	\thebar}

\newcommand{\Q}{\mathbb Q}

\usepackage{authblk}
\newcommand{\arc}[1]{\draw[black] #1 arc (0:180:.5);}
\newcommand{\tic}[1]{\draw #1+(-.27,0) node[above right] #1 {\tiny /};}
\newcommand{\ticone}[1]{\draw #1+(.27,0) node #1 {\tiny /};}
\begin{document}

\title{Triangular Ladders $P_{d,2}$ are $e$-positive}
\author{Samantha Dahlberg}
\affil{Arizona State University}
 \date{}


\newcounter{foo}
\newcommand{\rtask}[1]{\refstepcounter{foo}\label{#1}}

\maketitle

\abstract{ In 1995 Stanley conjectured that the chromatic symmetric functions of the graphs $P_{d,2}$, which we call triangular ladders, were $e$-positive. In this paper we confirm this conjecture, which is also an unsolved case of the celebrated $(3+1)$-free conjecture. Our method is to follow the  generalization of the chromatic symmetric functions by  Gebhard and Sagan to symmetric functions in non-commuting variables. These functions satisfy a deletion-contraction property unlike the chromatic symmetric function in commuting variables. We do this by proving a new signed combinatorial formula for \emph{all} unit interval graphs on the basis of elementary symmetric functions. Then we prove $e$-positivity for triangular ladders by very carefully defining a sign-reversing involution on our signed combinatorial formula. This leaves us with certain positive terms and further allows us to expand on an already-known family of $e$-positive graphs by Gebhard and Sagan. 
}

\section{Introduction}

The chromatic symmetric function of a simple graph  $G$ defined by Richard Stanley~\cite{Stan95}, $X_G$, is a generalization of the chromatic polynomial defined by Birkoff~\cite{B12}, $\chi_G$, and has received a lot of attention as of late. These symmetric functions carry many properties over from their chromatic polynomials including the number of acyclic orientations~\cite{Stan95}, but do not satisfy a useful deletion-contraction property that  chromatic polynomials do. However, they do have connections to representation theory and algebraic geometry~\cite{HP}, which has been a further motivation in their study and particularly behind the study of their $e$-positivity that is the ability to write the function $X_G$ as a non-negative sum of elementary symmetric functions, and Schur-positivity that is the ability to write $X_G$ as a non-negative sum of Schur  functions. In 1995 Stanley~\cite{Stan95}  conjectured that if a poset is $(3+1)$-free then its incomparability graph is $e$-positive, which is equivalent to the Stanley-Stembridge conjecture in 1993~\cite{SS93}. This conjecture has been reduced to showing that  the incomparability graphs of  $(3+1)$ and $(2+2)$-free posets are $e$-positive by Guay-Paquet~\cite{GP}. These types of graphs are known as unit interval graphs and have a connection to Jacobi-Trudi matrices~\cite{SS93}. Gasharov~\cite{G99} has proven that  the incomparability graph of $(3+1)$-free poset is Schur-positive, which is weaker than the full conjecture since $e$-positivity implies Schur-positivity. 

There have been some partial results on this conjecture. The path and cycle graphs have been shown to be $e$-positive by Stanley in 1995~\cite{Stan95} with a full description of their coefficients in~\cite{W98}. Coefficients of other graphs have been studied in~\cite{P, PS}. Other works have focused on finding graph properties relating to $e$-positivity  with an emphasis on induced subgraphs~\cite{DFvW, HHT, T18}. Shareshian and Wachs~\cite{SW16} defined a generalization of the chromatic symmetric function in the space of quasi-symmetric functions and have generalized the $(3+1)$-free conjecture as well as conjectured that these quasi-symmetric functions are $e$-unimodal. This has further been generalized by Ellzey~\cite{E17} to circular indifference graphs. One new family of unit interval graphs has been proven to be $e$-positive by Cho and Huh~\cite{CH} with an old family by Stanley proven to be $e$-positive with a new proof.

In this paper we prove that the graphs $P_{d,2}$ are $e$-positive, which are specifically mentioned in Stanley's original 1995 paper [p190, 14] where he wrote  
\begin{center}``It remains open whether $P_{d,2}$ is $e$-positive.'' \end{center}
In order to do this we follow a different generalization of $X_G$ by  Gebhard and Sagan~\cite{GS01}  to symmetric functions in non-commuting variables, which does satisfy a deletion-contraction property. Gebhard and Sagan in their paper prove more graphs have an $e$-positive $X_G$ by semi-symmetrizing their chromatic symmetric functions in non-commuting variables.  We use ideas in their paper and expand on
their proven family of $e$-positive graphs including all $P_{d,2}$, which we call \emph{triangular ladders}.  
We do this by proving a new signed combinatorial formula for \emph{all} unit interval graphs in 
the basis of elementary symmetric functions. 
Then we prove $e$-positivity for triangular ladders by very carefully defining a sign-reversing involution on our signed combinatorial formula, which leaves us with certain positive terms.

In Section \ref{background} we describe the necessary background we need to derive our signed combinatorial  formula in the elementary basis including the definition of unit interval graphs and Gebhard and Sagan's deletion-contraction property in non-commuting variables. In Section~\ref{sec:formulaSetUp} we derive our signed combinatorial  formula in the elementary basis for any unit  interval graph. Our method is to  repeatedly use the deletion-contraction property on our graphs until we arrive at a single vertex and then reinterpret the coefficients in a combinatorial manner using arc diagrams with arc markings, vertex labels and vertex markings. In Section~\ref{sec:TL} we apply our signed combinatorial  formula to triangular ladders and carefully define a sign-reversing involution in order to prove these graphs are $e$-positive. Lastly, in Section~\ref{sec:More} we use the sign-reversing involution along with results by Gebhard and Sagan to show how we can combine complete graphs and triangular ladders to form more $e$-positive graphs. 

\section{Background}
\label{background}

In this section we will go over the necessary background needed derive our signed combinatorial  formula in the elementary basis. 
Throughout this paper we will work with simple graphs $G$ with labeled vertices and vertex labels in $[n]=\{1,2,\ldots, n\}$, and we particularly focus on  labeled unit  interval graphs. An {\it unit interval graph} on vertices in $[n]$ is a graph formed by a collection of intervals $[a_1,b_1]$, $[a_2,b_2]$, $\ldots$, $[a_l,b_l]$ where $a_k\leq b_k$ are in $[n]$ and where we define $[a,b]=\{a,a+1,\ldots, b\}$. The unit interval graph $G$ with the given intervals will have all possible edges from vertex $i$ to $j$ whenever $i,j\in[a_k,b_k]$ for some $k$. There are many equivalent ways to define unit interval graphs, with some proofs between the equivalent definitions in~\cite{E17} by Ellzey.   In the literature there are special families of unit interval graphs $P_{n,k}$, which are formed from the intervals $[1,1+k]$, $[2,2+k]$, $\ldots$, $[n-k,n]$ and is the notation Stanley uses in his paper~\cite{Stan95}. This notation defines many well-known families of graphs including the {\it complete graphs},  $K_n=P_{n,n-1}$, and the {\it paths}, $P_n=P_{n,1}$. Since in later sections we will be focusing on one particular family when $k=2$, we will call the $P_{n,2}$ the {\it triangular ladders}, $TL_n$,  to help the reader keep in mind our object of interest. 
Though it is not necessary for the definition,  will write tend to write the intervals as $[a_1,1],[a_2,2],\ldots,[a_n,n]$, which may be redundant. In Figure~\ref{fig:TL7} we draw the unit interval graph for intervals $[1,1],[1,2],[1,3],[2,4],[3,5],[4,6],[5,7]$.

\begin{figure}
\begin{center}
\begin{tikzpicture}
\draw[black] (1,0) arc (0:180:.5);
\draw[black] (2,0) arc (0:180:.5);
\draw[black] (3,0) arc (0:180:.5);
\draw[black] (1.5,0) arc (0:180:.5);
\draw[black] (2.5,0) arc (0:180:.5);
\draw[black] (0,0)--(3,0);
\filldraw[black] (0,0) circle [radius=2pt] node[below] {$1$};
\filldraw[black] (.5,0) circle [radius=2pt] node[below] {$2$};
\filldraw[black] (1,0) circle [radius=2pt] node[below] {$3$};
\filldraw[black] (2,0) circle [radius=2pt] node[below] {$5$};
\filldraw[black] (3,0) circle [radius=2pt] node[below] {$7$};
\filldraw[black] (1.5,0) circle [radius=2pt] node[below] {$4$};
\filldraw[black] (2.5,0) circle [radius=2pt] node[below] {$6$};
\end{tikzpicture}
\begin{tikzpicture}
\coordinate (A) at (0,0);
\coordinate (B) at (1,0);
\coordinate (C) at (1,1);
\coordinate (D) at (2,0);
\coordinate (E) at (2,1);
\coordinate (F) at (3,0);
\coordinate (G) at (3,1);
\draw[black] (A)--(B)--(C)--(D)--(E)--(F)--(G);
\draw[black] (A)--(C)--(E)--(G);
\draw[black] (B)--(D)--(F);
\filldraw[black] (A) circle [radius=2pt] node[below] {$1$};
\filldraw[black] (B) circle [radius=2pt] node[below] {$2$};
\filldraw[black] (C) circle [radius=2pt] node[above] {$3$};
\filldraw[black] (E) circle [radius=2pt] node[above] {$5$};
\filldraw[black] (G) circle [radius=2pt] node[above] {$7$};
\filldraw[black] (D) circle [radius=2pt] node[below] {$4$};
\filldraw[black] (F) circle [radius=2pt] node[below] {$6$};
\end{tikzpicture}
\hspace{1cm}
\begin{tikzpicture}
\coordinate (A) at (0,0);
\coordinate (B) at (1,-.5);
\coordinate (C) at (1,.5);
\coordinate (D) at (2,0);
\coordinate (E) at (3,-.5);
\coordinate (F) at (3,.5);
\coordinate (G) at (4,0);
\draw[black] (A)--(B)--(C)--(D)--(E)--(F)--(G);
\draw[black] (A)--(C)--(E)--(G);
\draw[black] (B)--(D)--(F);
\draw[black] (A)--(D);
\filldraw[black] (A) circle [radius=2pt] node[below] {$1$};
\filldraw[black] (B) circle [radius=2pt] node[below] {$2$};
\filldraw[black] (C) circle [radius=2pt] node[above] {$3$};
\filldraw[black] (E) circle [radius=2pt] node[below] {$5$};
\filldraw[black] (G) circle [radius=2pt] node[above] {$7$};
\filldraw[black] (D) circle [radius=2pt] node[below] {$4$};
\filldraw[black] (F) circle [radius=2pt] node[above] {$6$};
\end{tikzpicture}
\end{center}
\caption{The triangular ladder graph $TL_7$ on the left and center and $K_4\cdot TL_4$ on the right.}
\label{fig:TL7}
\end{figure}
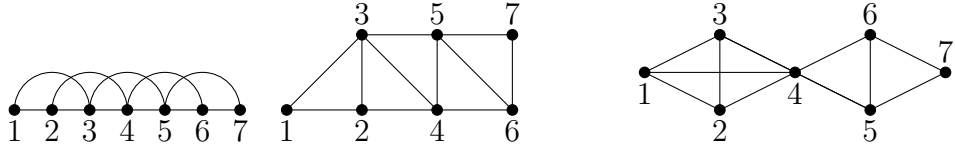

Before we introduce graphs colorings and the chromatic symmetric function in commuting or non-commuting variables let us review the algebra these objects exist in. The {\it algebra of symmetric functions in non-commuting variables} is a sub-algebra of $\Q[[x_1,x_2,\ldots]]$ where variables $x_1,x_2,\ldots$ do not commute and all $f\in\NCSym$ are unchanged by permuting subscripts. 
There are several classical bases that generate NCSym, including the power-sum basis and the elementary basis, all of which are indexed by set partitions. A {\it set partition}, $\pi=B_1/B_2/\cdots/B_l$ of $[n]$, denoted $\pi\vdash[n]$, is a collection of non-empty disjoint subsets $B_i\subseteq [n]$ called {\it blocks} that union to form the full set $[n]$. Given $\pi\vdash[n]$ and $\sigma\vdash[m]$ define $\pi|\sigma$ to be the set partition of $[n+m]$ we get from all blocks of $\pi$ together with all blocks of $\sigma$ except all elements in the blocks of $\sigma$ are increased by $n$. For example, $13/2|14/23=13/2/47/56$. Rosas and Sagan~\cite{RS06} define all the classical functions and give conversion formulas between them. In this paper we will only work with the power-sum and elementary functions. Given $\pi\vdash[n]$ the {\it  power-sum function in non-commuting variables}, $p_{\pi}$, in NCSym  is
$$p_{\pi}=\sum_{(i_1,i_2,\ldots, i_n)}x_{i_1}x_{i_2}\cdots x_{i_n},$$
which is summed over tuples $(i_1,i_2,\ldots, i_n)$ of positive integers where $i_j=i_k$ if $j$ and $k$ are in the same block in $\pi$. Given $\pi\vdash[n]$ the {\it  elementary symmetric function in non-commuting variables}, $e_{\pi}$, in NCSym  is
$$e_{\pi}=\sum_{(i_1,i_2,\ldots, i_n)}x_{i_1}x_{i_2}\cdots x_{i_n},$$
which is summed over tuples $(i_1,i_2,\ldots, i_n)$ of positive integers where $i_j\neq i_k$ if $j$ and $k$ are in the same block in $\pi$. 
These two bases are multiplicative, meaning that if $\pi=\pi_1|\pi_2$ then 
$$p_{\pi_1|\pi_2}=p_{\pi_1}p_{\pi_2} \text{ and } e_{\pi_1|\pi_2}=e_{\pi_1}e_{\pi_2}.$$
 Bergeron et. al.~\cite{B+06} proved the power-sum equality, but the elementary symmetric function equality comes quickly from the change of basis formula by Rosas and Sagan~\cite{RS06}, which is 
\begin{equation}
e_{\pi}=\sum_{\sigma\leq \pi} \mu(\hat 0, \sigma)p_{\sigma}
\label{eq:emulti}
\end{equation}
where $\sigma\leq \pi$ comes from the the poset of set partitions ordered by  refinement with $\hat 0=1/2/\cdots /n$ being the smallest element in the poset. For more information about this poset and its M\"{o}bius function $\mu$ see Stanley's book~\cite{S11}.  Since the multiplicativity of the elementary basis is vital to many of our proofs we prove it here. 
\begin{lemma}
For $\pi\vdash[n]$ with $\pi=\pi_1|\pi_2$ where $\pi_1$ and $\pi_2$ are non-empty set partitions we have
$e_{\pi_1|\pi_2}=e_{\pi_1}e_{\pi_2}.$
\end{lemma} 
\begin{proof}
Using the change of basis formula in equation~\eqref{eq:emulti} we have
$$e_{\pi_1}e_{\pi_2}=\sum_{\sigma_1\leq \pi_1} \mu(\hat 0, \sigma_1)p_{\sigma_2}\sum_{\sigma\leq \pi_2} \mu(\hat 0, \sigma_2)p_{\sigma_2}=\sum_{\sigma_1\leq \pi_1} \sum_{\sigma_2\leq \pi_2} \mu(\hat 0, \sigma_1)\mu(\hat 0, \sigma_2)p_{\sigma_1|\sigma_2}$$
using the fact that the power-sum basis is multiplicative. The M\"{o}bius function for the set partition poset ordered by refinement is also multiplicative in that $ \mu(\hat 0, \sigma_1)\mu(\hat 0, \sigma_2)=\mu(\hat 0, \sigma_1|\sigma_2)$ and $\sigma\leq \pi_1|\pi_2$ if and only if $\sigma=\sigma_1|\sigma_2$ with $\sigma_1\leq \pi_1$ and $\sigma_2\leq \pi_2$. Using the change of basis formula again we get our result. 
\end{proof}

Though most of our work in this paper will deal with symmetric functions in NCSym, the results we want to prove deal with $\Lambda$, the algebra we get by letting the variables in NCSym commute. Define $\rho:\NCSym\rightarrow \Lambda$ to be the commuting map where $f\in\NCSym$ is mapped to $f$ but we let the variables commute.  The algebra $\Lambda$ is indexed by {\it integer partitions}, $\lambda=\lambda_1\lambda_2\ldots \lambda_l$, which is a weakly decreasing list of positive integers where if all the $\lambda_i$ sum to $n$ we write $\lambda\vdash n$. There are similar classical bases in $\Lambda$ like the elementary symmetric functions and power-sum functions. The {\it  $i$th elementary symmetric function in commuting variables} is 
$$e_i=\sum_{i_1<i_2<\cdots<i_j}x_{i_1}x_{i_2}\cdots x_{i_j}$$ and 
the {\it  $i$th power-sum function in commuting variables} is 
$$p_i=x_1^i+x_2^i+x_3^i+\cdots$$
where as for an integer partition $\lambda=\lambda_1\lambda_2\ldots \lambda_l$ we define the {\it elementary symmetric function}, $e_{\lambda}$, and the {\it power-sum function}, $p_{\lambda}$, to be
$$e_{\lambda}=e_{\lambda_1}e_{\lambda_2}\cdots e_{\lambda_l}\text{ and }p_{\lambda}=p_{\lambda_1}p_{\lambda_2}\cdots p_{\lambda_l}.$$
These functions do have a close relationship to their relatives in NCSym. To relate them we define for a set partition $\pi\vdash[n]$ the integer partition $\lambda(\pi)\vdash n$, which we form by taking the sizes of all the blocks in  $\pi$. For example, $\lambda(134/25/67)=322$. Rosas and Sagan~\cite{RS06} showed that  $\rho(p_{\pi})=p_{\lambda(\pi)}$ and $\rho(e_{\pi})=\pi!e_{\lambda(\pi)}$ where $\pi!=\lambda(\pi)!=\lambda_1!\lambda_2!\cdots\lambda_l!$ and $\pi\vdash [n]$ is a set partition. 
We will call a function $f\in\Lambda$ {\it $e$-positive} if $f$ can be written as a non-negative sum of elementary symmetric functions. 

The symmetric functions in NCSym we study are defined from a graph $G$ and its proper colorings. A {\it proper coloring} $\kappa$ of a graph $G$ with vertex set $V$   is a function
$$\kappa : V\rightarrow \{1,2,\ldots\}$$such that if $v_1, v_2 \in V$ form an edge, then $\kappa(v_1)\neq \kappa(v_2)$. The {\it chromatic symmetric function in non-commuting variables} is defined to be
$$Y_G = \sum _\kappa x_{\kappa(v_1)}x_{\kappa(v_2)}\cdots x_{\kappa(v_n)}$$
where the sum is over all proper colorings $\kappa$ of $G$ where variables don't commute. If we let the variables commute then we get the {\it chromatic symmetric function in commuting variables}, which we denote $X_G=\rho(Y_G)$. We will call a graph $G$ itself {\it $e$-positive} if $X_G$ is $e$-positive. For example
$$Y_{G_n}=e_{12\cdots n}\text{ and }\rho(Y_{K_n})=X_{K_n}=n!e_n,$$
so all complete graphs are $e$-positive. 

The main result of the paper is proving a certain family of graphs is $e$-positive. This family of graphs is formed by combining complete graphs and triangular ladders in a certain way.  Given a graph $G$ with labels in $[n]$ and a graph $H$ with labels in $[m]$ we define their {\it concatenation} to be the graph $G\cdot H$  on vertices $[n+m-1]$ where the graph on the first $n$ vertices  is isomorphic to $G$ and the graph on the last  $m$ vertices  is isomorphic to $H$, where by isomorphic we mean that the underlying graphs are isomorphic and the vertex labels are in the same relative order. See Figure~\ref{fig:TL7} for an example. 

Though the chromatic symmetric function $X_G$ in commuting variables doesn't satisfy a deletion-contraction property, the $Y_G$, shown by Gebhard and Sagan~\cite{GS01}, do satisfy a deletion-contraction property. We define the {\it deletion} of an edge $\epsilon$ of $G$, $G\setminus \epsilon$, to be the graph $G$ with edge $\epsilon$ removed. Though contraction can be defined for any edge, for our purposes and for simplicity, we will only define the {\it contraction} of an edge $\epsilon$ that is between vertices $j$ and $n$.
The contracted graph, $G/\epsilon$, is the graph $G$ where we identify the vertices $j$ and $n$ and remove any multi-edges or loops created. In order to handle the idea of edge contraction in terms of functions in NCSym we   define an induced function. 
Define the {\it induced} monomial to be 
$$x_{i_1}x_{i_2}\dots x_{i_j} \dots x_{i_{n-1}}\uparrow_j^n=x_{i_1}x_{i_2}\dots x_{i_j} \dots x_{i_{n-1}}x_{i_j}$$
 where we make an extra copy of the $j$th variable at the end and extend this definition linearly. Given an integer partition  $\pi\vdash [n-1]$ we define for positive $j<n$ that $\pi\oplus_j n\vdash [n]$ is the integer partition $\pi$ but  we place $n$ in the same block as $j$. For example, $14/23\oplus_4 5=145/23$. We extend this definition to $j=n$ by letting $\pi\oplus_n n=\pi/n$. Gebhard and Sagan~\cite{GS01} offer that it is not hard to see  for $\pi\vdash [n-1]$ and $j<n$ that 
\begin{equation}
p_{\pi}\uparrow_j^n = p_{\pi\oplus_j n}.
\label{eq.induced}
\end{equation}
For ease of notation later we define $$p_{\pi}\uparrow_n^n=p_{\pi\oplus_n n}=p_{\pi}p_1$$
and extend linearly. 

\begin{prop}[Deletion-Contraction,  Gebhard and Sagan~\cite{GS01} Proposition 3.5] 
For $G$ with vertices $V=[n]$ and an edge $\epsilon$ between vertices $j$ and $n$ we have 
$$Y_G=Y_{G\setminus \epsilon}-Y_{G/\epsilon}\uparrow_j^n.$$
\label{prop.DeletionContraction}
\end{prop}

Though there are nice formulas for inducing in the power-sum  basis, the formula for the elementary basis has many terms. However, after symmetrizing, many of these terms cancel out. Gebhard and Sagan defined equivalence classes on set partitions that enable us to partially symmetrize functions. We will say two set partitions $\pi$ and $\sigma$ are $\pi\sim \sigma$ if
\begin{enumerate}
\item $\lambda(\pi)=\lambda(\sigma)$ and
\item if $A$ and $B$ are blocks of $\pi$ and $\sigma$ respectively and if $n\in A$ and $n\in B$ then $|A|=|B|$.
\end{enumerate}
Define 
\begin{equation}
(\pi)=\{\sigma:\sigma\sim \pi\}.
\label{eq:equivclass}
\end{equation}
Say two functions $f,g\in\NCSym$ are  $f\equiv_n g$  if the sum of coefficients in the elementary basis in the same equivalence classes are the same. For example, consider the chromatic symmetric function of the path graph on three vertices calculated by Gebhard and Sagan in~\cite{GS01}. We have
$$Y_{P_3}=\frac{1}{2}(e_{12/3}-e_{13/2}+e_{1/23}+e_{123})\equiv_3  \frac{1}{2}(e_{12/3}+e_{123})$$
because $13/2\sim 1/23$. 
Our study in this paper is about whether graphs $G$ themselves are $e$-positive, which is dealing with $X_G$ in full commuting variables.  Though it is an abuse of terminology, our goal is to show that $Y_G$ is $e$-positive after  partially symmetrizing variables along the lines of these equivalence classes. To formalize this, we call a function $f\in\NCSym$ {\it semi-symmetrized $e$-positive} if $f\equiv_n g$ for some $g\in\NCSym$ that can be written as  a non-negative sum of elementary symmetric functions in non-commuting variables. We call a graph $G$ {\it semi-symmetrized $e$-positive} if $Y_G$ is semi-symmetrized $e$-positive.  It follows that if $Y_G$ is semi-symmetrized $e$-positive then certainly $\rho(Y_G)=X_G$ is $e$-positive and $G$ is e-positive. This makes semi-symmetrized $e$-positivity  a stronger condition than $e$-positivity.

There are two propositions by Gebhard and Sagan that are essential to our proofs. One is a formula for inducing elementary symmetric functions and the other is a relabeling proposition. 

\begin{prop}[Gebhard and Sagan~\cite{GS01} Corollary 6.1] For $\pi\vdash [n-1]$, $j<n$ and $b$ the size of the block in $\pi$ containing $n-1$ we have
$$e_{\pi}\uparrow_j^n\equiv_n \frac{1}{b}(e_{\pi/n}-e_{\pi\oplus_j n}).$$
\label{prop:gebsagInduce}
\end{prop}

The relabeling proposition considers how permuting vertex labels affects the chromatic symmetric function in $\NCSym$. Given  $\delta\in \fS_n$ and $f\in \NCSym$ define $\delta\circ f$ to be the function after we permute the placements of the variables, rather than the subscripts. For example, having $\delta = 213$ acting on $x_1x_2x_1$ means we switch the first two variables so $\delta \circ x_1x_2x_1=x_2x_1x_1$. Also define for a graph $G$ on vertices labeled with $[n]$ a new graph $\delta(G)$, which is $G$ but we permute the labels of the vertices. We similarly define $\delta(\pi)$ for $\pi\vdash[n]$ by permuting the elements in $[n]$. The following is Gebhard and Sagan's relabeling proposition. 
\begin{lemma}[Relabeling Proposition, Gebhard and Sagan~\cite{GS01} Proposition 3.3] For a graph $G$ with distinct vertex labels in $[n]$ and $\delta\in\fS_n$,
$$Y_{\delta(G)}=\delta\circ Y_G.$$
\label{prop.relabeling}
\end{lemma}
We will also need a slight generalization of Gebhard and Sagan's result that easily follows from the relabeling proposition~\ref{prop.relabeling}.

\begin{lemma}[Gebhard and Sagan~\cite{GS01} Lemma 6.6] If $f\equiv_n g$ and $\delta\in \fS_n$ with $\delta(n)=n$ then
$$\delta\circ f \equiv_n \delta\circ g.$$
\end{lemma}

\section{Formula in the elementary basis}
\label{sec:formulaSetUp}

In this section we develop a new formula for unit intervals graphs in the elementary basis in terms of signed combinatorial objects involving labeled arc diagrams with arc markings and vertex markings. The idea is that we will delete and contract our graph down to a single vertex and then induce the chromatic symmetric function on one vertex back until we get the full function of our unit interval graph. If we consider this inducing in terms of the power-sum basis we will arrive at an example of Stanley's broken-circuit theorem (\cite{Stan95} Theorem 2.9). Since our interest is in the elementary basis, we will use Gebhard and Sagan's~\cite{GS01} formula in Proposition~\ref{prop:gebsagInduce} for inducing elementary symmetric functions to a find a signed combinatorial interpretation of these coefficients. In Section~\ref{sec:TL} we will show in the case of triangular ladders that we can define a sign-reversing involution on our signed combinatorial objects, which will prove that  a new family of graphs, the triangular ladders, is $e$-positive.
We will develop the signed combinatorial formula in stages first determining a recursive formula for the chromatic symmetric function of unit interval graphs.

\begin{thm}
Given a unit interval graph on $n$ vertices with intervals $[a_1,1],[a_2,2],\ldots,[a_n,n]$ let $G'$ be the same graph on $n-1$ vertices after removing vertex $n$. Then
$$Y_G=Y_{G'}Y_{K_1}-\sum_{i=a_n}^{n-1}Y_{G'}\uparrow_i=Y_{G'}\uparrow_n^n-\sum_{i=a_n}^{n-1}Y_{G'}\uparrow_i^n.$$
\label{thm:Ginduce}
\end{thm}
\begin{proof}
We will prove this by inducting on $n$, the number of vertices, and $m$, the number of edges in a unit interval graph $G$ defined by the intervals $[a_1,1],[a_2,2],\ldots,[a_n,n]$. The base case is $K_1$ when $n=1$ and $m=0$. This case is easy to see since everything is equal to $e_1$. 

Now assume that $G$ is a unit interval graph on $n>1$ vertices with $m$ edges. We will assume that any unit interval graph $H$ with $\bar n\leq n$ vertices and $\bar m\leq m$ edges with either $\bar n<n$ or $\bar m<m$ satisfies the above formula. Define $G'$ to be the graph $G$, but we remove vertex $n$, so $G'$  satisfies the formula. 
If $a_n=n$ then $G$ is the disjoint union of $G'$ and $K_1$. It is not hard to see that $G$ satisfies the formula because  $Y_G=Y_{G'}Y_{K_1}$.   Say that instead $a_n<n$. By deletion-contraction in Proposition~\ref{prop.DeletionContraction} using the edge $\epsilon$ between vertices $a_n$ and $n$ we have
$$Y_G=Y_{G-\epsilon}-Y_{G/\epsilon}\uparrow_{a_n}^n.$$
Note that $G/\epsilon$ is $G'$. Also, note that $G-\epsilon$ is also a unit interval graph with all the same intervals as $G$, but the interval $[a_n,n]$ changes to $[a_n+1,n]$. If we remove vertex $n$ from $G-\epsilon$ then we get $G'$. Since $G-\epsilon$ has less edges than $G$  by induction we can say
$$Y_{G-\epsilon}=Y_{G'}Y_{K_1}-\sum_{i=a_n+1}^{n-1}Y_{G'}\uparrow_i^n.$$
Putting everything all together we get the equation in this proposition. 
\end{proof}

We will continually use the formula in Theorem~\ref{thm:Ginduce} until we are only inducing from $K_1$. This gives us 
\begin{equation}Y_G=\sum_{i_n=a_n}^{n}\cdots \sum_{i_3=a_3}^{3}\sum_{i_2=a_2}^{2}(-1)^{|\{i_j\neq j\}|}Y_{K_1}\uparrow_{i_2}^{2}\uparrow_{i_3}^{3}\cdots   \uparrow_{i_n}^{n}.
\label{eq:Ginduce}
\end{equation}
We will represent each series of inducings $\uparrow_{i_2}^{2}\uparrow_{i_3}^{3}\cdots   \uparrow_{i_n}^{n}$ with an arc diagram. An {\it arc diagram} is a drawing on $n$ vertices in a line numbered from left to right together with a collection of arcs $(i,j)$ with $i<j$ representing an edge from $i$ to $j$. 
A series of inducings like $\uparrow_{i_2}^{2}\uparrow_{i_3}^{3}\cdots   \uparrow_{i_n}^{n}$ will be represented by the arc diagram on $n$ vertices with arcs $(i_2,2),(i_3,3),\ldots, (i_n,n)$ where if $i_j=j$ there is no arc, but we may list non-arcs for notational ease.  Define an arc $(i,j)$  to be a {\it left arc} of $j$. The collection of arc diagrams just described for a unit interval graph $G$ with intervals $[a_1,1],[a_2,2],\ldots,[a_n,n]$ are those where
\begin{itemize}
\item all vertices have at most one left arc and
\item if we have an arc $(i,j)$ then $i,j\in [a_k,k]$ for some $k$. 
\end{itemize}
Define this set of arc diagrams  to be $\cA(G)$. Note that the sign in equation~\eqref{eq:Ginduce} is determined by $|\{j\neq i_j\}|$, which is precisely the number of arcs in the arc diagram. 
For an arc diagram $D\in \cA(G)$ define $a(D)$ to be the number of arcs in the arc diagram $D$. 
See Figure~\ref{fig:arcEx} for an example. 
We will re-represent the series of inducings $\uparrow_{i_2}^{2}\uparrow_{i_3}^{3}\cdots \uparrow_{i_n}^{n}$, which is associated to some arc diagram $D$, to be $\uparrow_D$. 

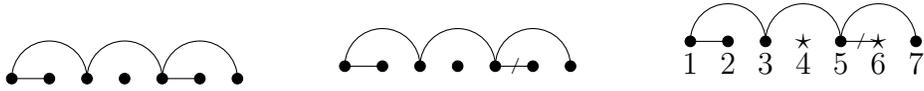
\begin{figure}
\begin{center}
\begin{tikzpicture}
\foreach \x in {0,...,1,2,3,4,5,6}
\filldraw[black] (\x/2,0) circle [radius=2pt];
\foreach \x in {}
\filldraw[black] (\x/2,0) node {$\star$};
\foreach \x in {0,4}{
\draw[black] (\x/2,0)--(\x/2+.5,0);}
\foreach \x in {2,4,6}{
\(\arc{(\x/2,0)} \)}
\foreach \x in {}{
\(\ticone{(\x/2,0)} \)}
\end{tikzpicture}
\hspace{1cm}
\begin{tikzpicture}
\foreach \x in {0,...,1,2,3,4,5,6}
\filldraw[black] (\x/2,0) circle [radius=2pt];
\foreach \x in {}
\filldraw[black] (\x/2,0) node {$\star$};
\foreach \x in {0,4}{
\draw[black] (\x/2,0)--(\x/2+.5,0);}
\foreach \x in {2,4,6}{
\(\arc{(\x/2,0)} \)}
\foreach \x in {4}{
\(\ticone{(\x/2,0)} \)}
\end{tikzpicture}
\hspace{1cm}
\begin{tikzpicture}
\foreach \x in {0,...,2,4,6}
\filldraw[black] (\x/2,0) circle [radius=2pt];
\foreach \x in {3,5}
\filldraw[black] (\x/2,0) node {$\star$};
\foreach \x in {0,4}{
\draw[black] (\x/2,0)--(\x/2+.5,0);}
\foreach \x in {2,4,6}{
\(\arc{(\x/2,0)} \)}
\foreach \x in {4}{
\(\ticone{(\x/2,0)} \)}
\draw (0.0,-.3) node {1};
\draw (0.5,-.3) node {2};
\draw (1.0,-.3) node {3};
\draw (1.5,-.3) node {4};
\draw (2.0,-.3) node {5};
\draw (2.5,-.3) node {6};
\draw (3.0,-.3) node {7};
\end{tikzpicture}
\end{center}
\caption{These are arc diagrams for the unit interval graph $G$ from intervals $[1,3],[2,4],[3,5],[4,6],[5,7]$. From left to right we have elements of $\cA(G)$, $\cA'(G)$ and $\cA'_L(G)$ where in all cases $a(D)=5$. On the left the associated set partition is $123567/4$ and the other two are $1235/4/67$. }
\label{fig:arcEx}
\end{figure}

\begin{prop}For a unit interval graph $G$,
$$Y_G=\sum_{D\in \cA(G)}(-1)^{a(D)}e_1\uparrow_D.$$
\label{prop:Ginduce2}
\end{prop}
\begin{proof}
This is equation~\eqref{eq:Ginduce}, but instead we represent the series of inducings as an arc diagram and use the fact that $Y_{K_1}=e_1$.
\end{proof}

Given any arc diagram $D\in\cA(G)$ we can use the connected components formed from the dots connected by arcs to form a set partition,  $\pi(D)$, if we make all the vertices in each connected component  a block. See Figure~\ref{fig:arcEx} for an example. 
One helpful fact about the series of inducings, which will be shown in the next lemma, is that 
$$p_1\uparrow_D=p_{\pi(D)}.$$

\begin{lemma}
For a unit interval graph $G$ and  arc diagram $D\in\cA(G)$ we have $p_1\uparrow_D=p_{\pi(D)}.$
As result, if $D_1$ and $D_2$ are two arc diagrams with $\pi(D_1)=\pi(D_2)$ then 
$$e_1\uparrow_{D_1}=e_1\uparrow_{D_2}.$$
\label{lem:anyOrder}
\end{lemma}
\begin{proof}
We will first prove that for an arc diagram $D$ associated to a unit interval graph $G$ with $n$ vertices that $p_1\uparrow_D=p_{\pi(D)}$ by inducting on $n$. The rest of the statement will follow from this fact. 

If $n=1$ then the only possible arc diagram is a single dot with no arcs and $\pi(D)=1$ so $p_{\pi(D)}=p_1$. Say that $n>1$. Consider an arc diagram $D$ with arcs $(i_2,2),(i_3,3),\ldots, (i_n,n)$. Let $\bar D$ be the arc diagram $D$, but with $n$ removed. By induction $p_1\uparrow_{\bar D}=p_{\pi(\bar D)}$. Then
$$p_1\uparrow_D=p_1\uparrow_{i_2}^{2}\uparrow_{i_3}^{3}\cdots \uparrow_{i_n}^{n}=p_{\pi(\bar D)}\uparrow_{i_n}^n.$$
If $n$ is in its own connected component then $D$ has no left arc at $n$ so $i_n=n$. Then $\pi(D)=\pi(\bar D)/n$ and we know that $p_{\pi(\bar D)}\uparrow_{i_n}^n=p_{\pi(\bar D)}p_1=p_{\pi(\bar D)/n}.$
Consider the case where $n$ is not in its own component in $\pi(D)$. That means $i_n<n$ and $n$ is in the same connected component as $i_n$. Further $\pi(D)=\pi(\bar D)\oplus_{i_n} n$. Similarly and using equation~\eqref{eq.induced}, 
$$p_1\uparrow_D=p_1\uparrow_{i_2}^{2}\uparrow_{i_3}^{3}\cdots \uparrow_{i_n}^{n}=p_{\pi(\bar D)}\uparrow_{i_n}^n=p_{\pi(\bar D)\oplus_{i_n} n}=p_{\pi(D)}.$$
Because $p_1\uparrow_D=p_{\pi(D)}$ we know that if $D_1$ and $D_2$ are two arc diagrams with $\pi(D_1)=\pi(D_2)$ then $p_1\uparrow_{D_1}=p_1\uparrow_{D_1}$. Because $p_1=e_1$ we have the rest of our result. 
\end{proof}

We will note that the formula in Proposition~\ref{prop:Ginduce2} is not particularly surprising, because if we instead used $Y_{K_1}=p_1$ and induced in the power-sum basis then we arrive at an example of Stanley's broken-circuit theorem \cite{Stan95}. Since we are particularly interested in  the elementary basis we will be using Gebhard and Sagan's formula in Proposition~\ref{prop:gebsagInduce} for inducing elementary symmetric functions, which will require us to semi-symmetrize the chromatic symmetric function. 

Because of Lemma~\ref{lem:anyOrder} we can define $e_1\uparrow_{\pi}$ for any set partition $\pi$ to be equal to $e_1\uparrow_D$ for any arc diagram $D$ with $\pi(D)=\pi$. Our method will be to continually use Gebhard and Sagan's inducing formula in the elementary basis to derive a signed combinatorial formula for a semi-symmetrized $Y_G$, which is distinct from the broken circuit theorem. First we must clarify exactly how we can continually use Gebhard and Sagan's inducing formula. Gebhard and Sagan prove this in their paper, but we mention the proof again because it is the backbone of our logic. 
\begin{lemma}[Gebhard and Sagan~\cite{GS01} Lemma 6.2]
For $f,g\in\NCSym$ if $f\equiv_{n-1} g$ then $f\uparrow_{n-1}^n\equiv_n g\uparrow_{n-1}^n.$
\label{lem:consecuative_induce}
\end{lemma}

\begin{proof}
Note that for $\pi,\sigma\vdash[n-1]$ if $\pi\sim \sigma$ then $\pi/n\sim \sigma/n$ and $\pi\oplus_{n-1}n\sim \sigma\oplus_{n-1}n$. This means if $e_\pi\equiv_{n-1}e_{\sigma}$ then $e_{\pi/n}\equiv_{n}e_{\sigma/n}$ and
$e_{\pi\oplus_{n}n}\equiv_{n-1}e_{\sigma\oplus_{n-1}n}$, which implies that $e_{\pi}\uparrow_{n-1}^n\equiv_n e_{\sigma}\uparrow_{n-1}^n$. Extending this linearly gives the result. 
\end{proof}

Note that when we induce an elementary symmetric function once, it is equivalent to the subtraction of two elementary symmetric functions after semi-symmetrizing. Since each inducing is associated to an arc in an arc diagram, we will need to keep track of these two possible terms for each inducing. We will do so by marking arcs with tic marks. 
Define $\cA'(G)$ to be the collection of arc diagrams $D\in \cA(G)$, but each arc will be decorated with a tic mark or left alone. See Figure~\ref{fig:arcEx} for an example.  We will consider each tic mark on arc $(i,j)$ to split the connected component into {\it pieces}, every dot to the left of $j$, but not including $j$, will be in a different piece then those to the right including $j$. For $D'\in \cA'(G)$ we define $\pi(D')$ to be the set partition formed by all these pieces the connected components  are broken into. We will also define $t(D')$ to be the number of tic marks on the diagram $D'$. For an arc diagram $D\in \cA(G)$ we define ${\cal T}(D)$ to be all $D'\in\cA'(G)$ but the underlying arc diagram is $D$ only. 

To get a formula for inducing $e_1$ in terms of elementary symmetric functions we will first consider the simple arc diagram $P_n$ on $n$ vertices with arcs $(1,2),(2,3),\ldots, (n-1,n)$. Each $D'\in \cA'(G)$ has an associated set partition $\pi(D')$, but each connected component also has an associated integer composition $\alpha=\alpha_1+\alpha_2+\cdots+\alpha_{l}\models n$ reading the sizes of the pieces between tic marks in one connected component in $D'$ from left to right. Where $\alpha=\alpha_1+\alpha_2+\cdots+\alpha_{l}\models n$ is an {\it integer composition} of $n$ if $\alpha_i\geq 1$ for all $i$ and the sum of the $\alpha_i$ is $n$. 

\begin{prop}For the set partition $[n]$ we have
$$e_1\uparrow_{[n]}\equiv_n\frac{1}{n!}\sum_{D'\in{\cal T}(P_n)}(-1)^{a(D')-t(D')}\alpha_{l}\binom{n}{\alpha_1,\ldots, \alpha_{l}}e_{\pi(D')}$$
where $\alpha=\alpha_1+\alpha_2+\cdots+\alpha_{l}\models n$ is the composition associated to the pieces of $D'$.
\label{prop:e_induce[n]} 
\end{prop}

\begin{proof}
We will prove this by inducting on $n$. If $n=1$ then $e_1\uparrow_{[1]}=e_1$ because $\uparrow_{[1]}$ means no inducing and we are done, so assume that $n>1$. By induction we know that 
$$e_1\uparrow_{[n-1]}\equiv_{n-1}\frac{1}{(n-1)!}\sum_{D'\in{\cal T}(P_{n-1})}(-1)^{a(D')-t(D')}\alpha_{l}\binom{n-1}{\alpha_1,\ldots, \alpha_{l}}e_{\pi(D')}$$
where $\alpha=\alpha_1+\alpha_2+\cdots+\alpha_{l}\models n$ is the composition associated to the pieces of $D'$.
By Lemma~\ref{lem:consecuative_induce}  this implies that
\begin{align*}
e_1\uparrow_{[n]}&=e_1\uparrow_{[n-1]}\uparrow_{n-1}^n\\
&\equiv_{n-1}\frac{1}{(n-1)!}\sum_{D'\in{\cal T}(P_{n-1})}(-1)^{a(D')-t(D')}\alpha_{l}\binom{n-1}{\alpha_1,\ldots, \alpha_{l}}e_{\pi(D')}\uparrow_{n-1}^n\\
&\equiv_{n-1}\frac{1}{(n-1)!}\sum_{D'\in{\cal T}(P_{n-1})}(-1)^{a(D')-t(D')}\alpha_{l}\binom{n-1}{\alpha_1,\ldots, \alpha_{l}}\frac{1}{\alpha_{l}}\left( e_{(\pi(D')/n)}-e_{(\pi(D')\oplus_{n-1}n)}\right).
\end{align*}
We will let the terms with $\pi(D')/n$ and $\pi(D')\oplus_{n-1}n$ be associated to $D'\in \cA'(P_n)$ with a tic mark on arc $(n-1,n)$ or no tic mark on arc $(n-1,n)$ respectively. Manipulating the signs, factorial and multinomial coefficient appropriately gives us the result. 
\end{proof}

We will use the formula for $e_1\uparrow_{[n]}$ that we just derived to find the formula for $e\uparrow D$ for a general arc diagram $D$. 

\begin{prop}For unit interval graph $G$ and  arc diagram $D\in \cA(G)$ we have
$$e_1\uparrow_{D}\equiv_n\frac{1}{n!}\sum_{D'\in{\cal T}(D)}(-1)^{a(D')-t(D')}
\prod_i\alpha^{(i)}_{l_i}\binom{n}{\alpha^{(1)}_1,\ldots, \alpha^{(1)}_{l_1},\alpha^{(2)}_1,\ldots, \alpha^{(2)}_{l_2},\ldots}e_{\pi(D')}$$
where $\alpha^{(i)}=\alpha^{(i)}_1+\alpha^{(i)}_2+\cdots+\alpha^{(i)}_{l_i}\models n$ are the compositions associated to the pieces of connected components of $D'$ where $\alpha^{(i)}_{l_i}$ is the size of the right-most piece. 
\label{prop:arc_tic_formula}
\end{prop}

\begin{proof} 
We will start the proof by considering the particular set partition $\pi=[{n_1}]|[{n_2}]|\cdots|[{n_{k}}]$ and prove this formula for $e_1\uparrow_{\pi}$. Then we will use the multiplicity of the power-sum basis and the formula for inducing the power-sum basis as well as the relabeling proposition~\ref{prop.relabeling} to show the formula for some particular arc diagram $D$ associated to $\pi$ before finally concluding this formula for a generic arc diagram $D$. 

First consider $\pi=[{n_1}]|[{n_2}]|\cdots|[{n_{k}}]$. Using the fact that  the elementary basis is multiplicative and Proposition~\ref{prop:e_induce[n]} we have that 
\begin{align*}
e_1\uparrow_{\pi}&=e_1\uparrow_{[{n_1}]}e_1\uparrow_{[{n_2}]}\cdots e_1\uparrow_{[{n_k}]}\\
&\equiv_n\prod_{j=1}^k\frac{1}{n_j!}\sum_{D'\in{\cal T}(P_{n_j})}(-1)^{a(D')-t(D')}\alpha_{l}\binom{n_j}{\alpha_1,\ldots, \alpha_{l}}e_{\pi(D')}\\
&\equiv_n\sum_{D'\in{\cal T}(P_{n_1}|P_{n_2}|\cdots|P_{n_k})}\prod_{j=1}^k\frac{1}{n_j!}(-1)^{a(D_j')-t(D_j')}\alpha^{(j)}_{l_j}\binom{n_j}{\alpha^{(j)}_1,\ldots, \alpha^{(j)}_{l_j}}e_{\pi(D_j')}\\
&\equiv_n\frac{1}{n!}\sum_{D'\in{\cal T}(P_{n_1}|P_{n_2}|\cdots|P_{n_k})}(-1)^{a(D')-t(D')}
\prod_i \alpha^{(i)}_{l_i}\binom{n}{\alpha^{(1)}_1,\ldots, \alpha^{(1)}_{\ell_1},\alpha^{(2)}_1,\ldots, \alpha^{(2)}_{l_2},\ldots}e_{\pi(D')}
\end{align*}
where we define for two diagrams $D_1$ on $n$ vertices and $D_2$ on $m$ vertices the diagram $D_1|D_2$ to be a diagram on $m+n$ vertices, which is isomorphic to $D_1$ on the first $n$ vertices and isomorphic to $D_2$ on the last $m$ vertices with no other arcs. 
Next we will consider a general set partition $\pi\vdash [n]$. We form an arc diagram $D$ such that we have an arc $(a,b)$ if $a$ and $b$ are listed consecutively in increasing order in a block of $\pi$. Also, there exists a permutation $\delta\in \fS_n$ such that $\delta(n)=n$ and if $(a,b)$ is an arc of $D$ then $\delta(a)=\delta(b)+1$. This makes $\delta(\pi)=[{n_1}]|[{n_2}]|\cdots|[{n_{k}}]$ for some $n_j$'s and $\delta(D)=P_{n_1}|P_{n_2}|\cdots|P_{n_k}$ where we permute the placement of the dots. In Lemma~\ref{lem:anyOrder} we showed that $e_1\uparrow_{D}=p_{\pi}$.
We know that 
\begin{align*}
\delta(p_{\pi})
&=p_{[n_1]}p_{[n_2]}\cdots p_{[n_k]}\\
&=e_1\uparrow_{[n_1]}e_1\uparrow_{[n_2]}\cdots e_1\uparrow_{[n_k]}\\
&\equiv_n\frac{1}{n!}\sum_{D'\in{\cal T}(P_{n_1}|P_{n_2}|\cdots|P_{n_k})}(-1)^{a(D')-t(D')}
\prod_i \alpha^{(i)}_{l_i}\binom{n}{\alpha^{(1)}_1,\ldots, \alpha^{(1)}_{l_1},\alpha^{(2)}_1,\ldots, \alpha^{(2)}_{l_2},\ldots}e_{\pi(D')}.
\end{align*}
Note that $\delta^{-1}(p_{\pi})=e_1\uparrow_D$. By the relabeling proposition~\ref{prop.relabeling} and taking $\delta^{-1}$ of the first and last part of the series of equivalences above we get the result for this special diagram $D$ associated to set partition $\pi$. 

Given any arc diagram $D\in \cA(g)$ there is a special arc diagram 
$D^*$ associated to $\pi(D)$ as we defined earlier. 
However, since we have already shown that $e_1\uparrow_{D}=e_1\uparrow_{D^{*}}$ 
when $\pi(D)=\pi(D^*)$ in Lemma~\ref{lem:anyOrder} we are done. 
\end{proof}

We are finally ready to introduce the signed combinatorial formula for a semi-symmetrized $Y_G$. The idea is to used tic'd arc diagrams and reinterpret the multinomial coefficient in Proposition~\ref{prop:arc_tic_formula} as labels on the vertices with some additional vertex markings. Given a unit interval graph $G$ define $\cA'_L(G)$ to be the collection of all arc diagrams in $\cA'(G)$ with a possible tic mark on each arc as well as a permutation label $\delta\in\fS_n$ on the vertices with $\delta(i)$ on vertex $i$. We want these labels to be increasing on each piece. Also, on each connected component we will mark one vertex in the most-right piece. We will use a star instead of a dot to show the vertex is marked. Meaning for an arc diagram $D'$ with possible tic marks that if $B=\{b_1,b_2,\ldots, b_k\}$ with $b_1<b_2<\cdots <b_k$ is a block of  $\pi(D')$   associated to a piece, then the permutation must have $\delta(b_1)<\delta(b_2)<\cdots <\delta(b_k)$. If this block was a right-most piece of a connected component then one of the vertices in $B$ is marked with a star.  See Figure~\ref{fig:arcEx} for an example.  Our signed combinatorial formula is as follows. 

\begin{thm}
For a unit interval graph $G$,
$$Y_G\equiv_n \frac{1}{n!}\sum_{D'\in \cA'_L(G)}(-1)^{t(D')}e_{\pi(D')}.$$
\label{thm:mainFormula}
\end{thm}
\begin{proof}
From  Propositions~\ref{prop:Ginduce2} and~\ref{prop:arc_tic_formula} we have 
\begin{align*}
Y_G&=\sum_{D\in \cA(G)}(-1)^{a(D)}e_1\uparrow_D\\
&\equiv_n\frac{1}{n!}\sum_{D\in \cA(G)}\sum_{D'\in{\cal T}(D)}(-1)^{t(D')}
\prod_i \alpha^{(i)}_{l_i}\binom{n}{\alpha^{(1)}_1,\ldots, \alpha^{(1)}_{l_1},\alpha^{(2)}_1,\ldots, \alpha^{(2)}_{l_2},\ldots}e_{\pi(D')}.
\end{align*}
A multinomial coefficient $\binom{n}{m_1,m_2,\ldots, m_k}$ can be combinatorially interpreted as a permutation increasing along the first $m_1$ indices, increasing along the next $m_2$ indices and so on.  We will combinatorially interpret the multinomial coefficient in the equation above  as  permutations $\delta\in\fS_n$ such that if $B=\{b_1,b_2,\ldots, b_k\}$ with $b_1<b_2<\cdots <b_k$ is a block of $\pi(D')$  associated to a piece then $\delta(b_1)<\delta(b_2)<\cdots <\delta(b_k)$. The multiplication by $\prod_i \alpha^{(i)}_{\ell_i}$ will be interpreted by marking  a vertex using a star in the right-most piece of each connected component, which each have size $\alpha^{(i)}_{l_i}$.
\end{proof}
\section{Triangular ladders}
\label{sec:TL}
In this section we will take the ideas from Section~\ref{sec:formulaSetUp} and apply them to triangular ladders, $TL_n$. 
 In~\cite{Stan95} Stanley used an iterative technique to solve for the chromatic symmetric function of a graph by solving a system of linear equations using Cramer's rule, and proved the $e$-positivity of the path and cycle graphs. In his paper he mentions that this technique does not work on $TL_n$ and mentions that proving the $e$-positivity remains open. In this section we prove $TL_n$ is semi-symmetrized $e$-positive and so $e$-positive for all $n$. Our method is to use our signed combinatorial formula from Theorem~\ref{thm:mainFormula} and define a sign-reversing involution on the associated signed combinatorial objects. 
In Section~\ref{sec:More} we will use the ideas from this section to prove all concatenations of complete graphs and triangular ladders  are $e$-positive.

Because triangular ladders are unit interval graphs we can obtain a number of corollaries from Section~\ref{sec:formulaSetUp}. 
\begin{cor}For $n\geq 2$,
$$Y_{TL_n}=Y_{{TL_{n-1}}}Y_{K_1}-Y_{{TL_{n-1}}}\uparrow_{n-1}^n-Y_{{TL_{n-1}}}\uparrow_{n-2}^n.$$
\end{cor}
\begin{proof}Follows quickly from Theorem~\ref{thm:Ginduce}. 
\end{proof}
Because we are only working with triangular ladders in this section, we will simplify the notation of $\cA(TL_n)$ as $\cA$
and $\cA'_L(TL_n)$ as $\cA'_L$.

\begin{cor}For $n\geq 1$,
$$Y_{TL_n}\equiv_n \frac{1}{n!}\sum_{D'\in \cA'_L}(-1)^{t(D')}e_{\pi(D')}.$$
\label{cor:TLsum}
\end{cor}
\begin{proof}Follows quickly from Theorem~\ref{thm:mainFormula}. 
\end{proof}

This section's focus will be to define a sign-reversing involution that can be applied to the signed summation in Corollary~\ref{cor:TLsum}. In order to define a sign-reversing involution we will first need to define a  {\it signed set}  $S$, which is a set  such that each $s\in S$ has an associated  {\it sign}, sign$(s)$, of $+1$ or $-1$. The elements assigned the value $+1$ are called the positive elements and the elements assigned $-1$ are the negative elements. Additionally, each element in our signed-set will have a {\it weight}, $\wt(s)$. A {\it sign-reversing involution} is a map $f:S\rightarrow S$ that is an involution,  $f\circ f = id$, and is {\it weight preserving}, $\wt(f(s))=\wt(s)$. The map $f$ also has to be {\it sign reversing} in that if $f(s)\neq s$ for $s\in S$  the sign associated to $s$ is opposite of the sign associated to $f(s)$. The consequence of such a map is a pairing between many elements in $S$ such that each pair shares the same weight and has a positive element and a negative element so that when we add the  signs of the pair together we get zero. Not all elements $s\in S$ will be part of a pairing. This happens when $f(s)=s$ and we call these elements {\it fixed points}.  This means regarding summations that
$$\sum_{s\in S}\sign(s)\wt(s)=\sum_{\text{fixed points $s$}}\sign(s)\wt(s)$$
and if all fixed points have positive sign, then we have proven our signed sum is actually a non-negative sum. 

Our sign-reversing involution, $\varphi:\cA'_L\rightarrow \cA'_L$, works with the signed set $\cA'_L$ where each element $D'\in\cA'_L$ has sign $\sign(D')=(-1)^{t(D')}$ and weight $\wt(D')=e_{\pi(D')}$ where the set partitions are viewed in the light of  their equivalence classes. This means that if $\pi(D'_1)\sim\pi(D'_2)$ then $D_1'$ has the same weight as $D_2'$  because $e_{D_1'}\equiv_n e_{D_2'}$. We will define this involution inductively, but we will first need some structure lemmas on the arc diagrams $D'\in \cA'_L$, whose underling unmarked and unlabeled arc diagrams are those in $\cA$. As a reminder, these   arc diagrams in $\cA$ are those  where
\begin{enumerate}
\item $D$ has at most one left arc at every vertex and
\item every arc is of length 1 or 2
\end{enumerate}
where the {\it length of an arc} $(i,j)$ is $j-i$. We define the {\it length of a diagram $D$}, $\ell(D)$,  to be the number of vertices minus one. 

The structure lemma will revolve around the idea of the {\it concatenation} of two arc diagrams $D_1$ on vertices $[n]$ and $D_2$ on vertices $[m]$, which we define to be $D_1\cdot D_2$, the arc diagram on $[n+m-1]$ where on the first $n$ vertices  we have $D_1$ and on the  last $m$ vertices we have $D_2$. Note that with our definition of length we have that 
$$\ell(D_1\cdot D_2)=\ell(D_1)+\ell(D_2), $$
which will be handy to keep in mind.

\begin{figure}
\begin{center}
\begin{tikzpicture}
\foreach \x in {1,...,4}
\filldraw[black] (\x/2,0) circle [radius=2pt];
\foreach \x in {}{
\draw[black] (\x/2,0)--(\x/2+.5,0);}
\foreach \x in {3,4}{
\(\arc{(\x/2,0)} \)}
\end{tikzpicture}
\hspace{1cm}
\begin{tikzpicture}
\foreach \x in {1,...,4}
\filldraw[black] (\x/2,0) circle [radius=2pt];
\foreach \x in {1}{
\draw[black] (\x/2,0)--(\x/2+.5,0);}
\foreach \x in {3,4}{
\(\arc{(\x/2,0)} \)}
\end{tikzpicture}
\hspace{1cm}
\begin{tikzpicture}
\foreach \x in {1,...,7}
\filldraw[black] (\x/2,0) circle [radius=2pt];
\foreach \x in {4,5}{
\draw[black] (\x/2,0)--(\x/2+.5,0);}
\foreach \x in {3,7}{
\(\arc{(\x/2,0)} \)}
\end{tikzpicture}
\end{center}
\caption{From left to write we have $L_3$, $C_3$ and $L_2\cdot L_1\cdot C_1\cdot C_2$.}
\label{fig:wedge}
\end{figure}
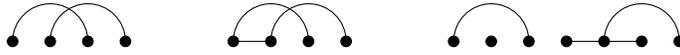

We will find that arc diagrams associated to the triangular ladders are concatenations of two kinds of arc diagrams. The first we will call an interlacing diagram. An {\it interlacing arc diagram}, IL, on $n\geq 2$ vertices will have arcs $(i,i+2)$ for all possible $i$. We define $L_m$ to be an IL diagram of length $m$. The second we will call an interconnecting diagram. An {\it interconnecting arc diagram}, IC, will be an IL diagram, but we include the arc $(1,2)$. We define $C_m$ to be an IC diagram of length $m$. See Figure~\ref{fig:wedge} for an example. In an IC diagram we have one connected component, $\pi(C_m)=[m+1]$. In an IL diagram we have two connected components whose sizes depend on whether the diagram is of odd length or even length.

\begin{lemma}
All diagrams $D\in\cA$ with two or more vertices are the concatenation of some IL and IC diagrams of lengths at least one, meaning
$$D=D_1\cdot D_2\cdot \cdots \cdot D_k$$
where each $D_i$ is an IL or IC diagram with $\ell(D_i)\geq 1$. 
\label{lem:arc_decomp}
\end{lemma}

\begin{proof}
We will prove this by inducting on the length of the arc diagram. Our base case is when $D$ is an arc diagram of $TL_2$ in which case $D$ is either an IL or IC diagram of length 1. Let $D$ be an arc diagram of $TL_n$ for $n>2$. Say $n$ doesn't have a left arc. Then on vertices $n-1$ and $n$ we have an IL diagram of length 1 and $D=\tilde D\cdot L_1$. The arc diagram $\tilde D$ has less than $n$ vertices so by induction we are done. Say that vertex $n$ has a left arc of length one. Then on vertices $n-1$ and $n$ we have an IC diagram of length 1 and $D=\tilde D\cdot C_1$. The arc diagram $\tilde D$ has less than $n$ vertices so by induction we are done. Lastly, consider the case where vertex $n$ has a left arc of length 2.  There will then be a largest integer $j$ where for all $i>j$ vertex $i$ has a left arc of length two, but vertex $j$ either has no left arc or a left arc of length one. In the first case on vertices $j-1$ through $n$ we have an IL diagram of length $n-j+1$ and $D=\tilde D\cdot L_{n-j+1}$ where  $\tilde D$ has length less than $n$. By induction we are done. In the second case on vertices $j-1$ through $n$ we have an IC diagram of length $n-j+1$ and $D=\tilde D\cdot C_{n-j+1}$ where  $\tilde D$ has length less than $n$. By induction we are done. 
\end{proof}

We will call the concatenation of $D\in\cA$ into IL and IC diagrams the {\it decomposition} of $D$. When talking about $D'\in\cA'_L$ the decomposition we are referring to is the decomposition of its underlying arc diagram $D\in\cA$.

We now have all the tools we need to define our sign-reversing involution. Again our signed-set is $\cA'_L$ where $D'\in\cA'_L$ has sign  $\sign(D')=(-1)^{t(D')}$ and weight $\wt(D')=e_{\pi(D')}$ where the weight is considered in light of the equivalence classes on set partitions defined by equation~\eqref{eq:equivclass} meaning that $\wt(D_1')=\wt(D_2')$ if and only if $\pi(D'_1)\sim\pi( D'_2)$. Further, we mean that  if $D_1'\mapsto D_2'$ then we want
\begin{enumerate}
\item  $\lambda(\pi(D'_1))=\lambda(\pi(D'_2))$ and
\item the sizes of the blocks containing $n$ has the same size.
\item Also, the label on the right-most vertex remains unchanged under the map and
\item  $D_1'$ has a star marking on its right-most vertex if and only if $D'_2$ does as well. 
\end{enumerate} 
These are the conditions we want our inductive map to satisfy. The motivations for the last two conditions will become apparent in a later discussion. 
In general, we will match $D'$ to another diagram that has an arc with a tic mark removed or an arc with a tic mark added. We will  be sure to fix the size of the piece attached to vertex $n$.  We will define this involution inductively on $n$. 

When $n=1$ we only have one kind of diagram, a single vertex labeled with 1 and marked with a star. Let diagram
\begin{center}
\begin{tikzpicture}
\foreach \x in {}
\filldraw[black] (\x/2,0) circle [radius=2pt];
\foreach \x in {0}
\filldraw[black] (\x/2,0) node {$\star$};
\foreach \x in {}{
\draw[black] (\x/2,0)--(\x/2+.5,0);}
\foreach \x in {}{
\(\arc{(\x/2,0)} \)}
\foreach \x in {}{
\(\ticone{(\x/2,0)} \)}
\draw (0.0,-.3) node {$1$};
\end{tikzpicture}
\begin{tikzpicture}
\draw (-1,.2) node {be a fixed point.};
\end{tikzpicture}
\end{center}
Say that $n>1$. By Lemma~\ref{lem:arc_decomp} we can write the underlying arc diagram $D'=P\cdot Q$ where $Q$ is an IL or IC diagram of length at least 1. However, before we continue on we must discuss the slight discrepancy between our decomposition of $D'$, which has vertex star markings and vertex labels, and the diagrams talked about in Lemma~\ref{lem:arc_decomp}, which do not have any labelings. When we decompose $D'=A\cdot B$ we will mostly be referring to the decomposition of the underlying arc diagram with possible tic marks ignoring vertex labels and markings. However, there will be times when we will need to refer to properties of $A$ as if it were also an element of $\cA'_L$. In this case we will keep the underlying vertex labels and think of $A$ as if its labels have been standardized to a permutation by replacing the $i$th smallest label with $i\geq 1$. If there is not a vertex star marking in the right-most piece of $A$ for any reason, we will think of $A$ as if the right-most vertex is marked with a star.  
This allows us to inductively map 
\begin{equation}
D'=P\cdot Q\mapsto \varphi(P)\cdot Q
\label{eq:varphi}
\end{equation}
as long as $P$ is not a fixed point.
Now we will have to discuss here how to reconnect $\varphi(P)$ to $Q$. First we keep the respective vertex labels that both $P$ and $Q$ had before the map, but they may be permuted in $\varphi(P)$.  As for the vertex star markings, if $\varphi(P)$ has its right-most vertex marked with a star, we will remove this if $Q$ contains a star-marked vertex in that piece, or that piece is not the right-most piece in the connected component. Reconnecting $\varphi(P)$ and $Q$ in terms of labeling makes sense as long as our map always preserves the label of the right-most vertex and makes sense in terms of vertex star markings if a diagram with the right-most vertex marked with a star is mapped to something with the same property. Note that by induction conditions 3 and 4 are satisfied. 
Also note that because this map fixes the size of the piece connected to the right-most vertex we can say that the piece containing $n$ has the same size before and after. Since also we maintain the underlying integer partition our output has the same weight as the input. This means conditions 1 and 2 are satisfied as well. Now we only have to consider cases where $P$ is a fixed point of $\varphi$. 

Before we classify the fixed points and discuss the map in these cases we need a few more definitions. Note that IL diagrams of length 1 naturally break our diagrams into sections.  We will define a {\it section} to be a diagram without any IL diagrams of length 1 in its decomposition. The right-most diagram in Figure~\ref{fig:wedge} has two sections. We will say a section {\it satisfies the IC-condition} if it starts with two consecutive IC diagrams of length at least one. See Figure~\ref{fig:markedexample}.

Recall that we will mark some our vertices with a star. Consider $D\in\cA'_L$ with no tic marks, which is just like how our fixed-points will be. Say that $D$ that ends in $Q=L_k$, $k\geq 1$. On vertex $n-1$ we have a right endpoint of a connected component. Because this connected component has only one piece, one vertex in this component will be marked with a star. We will define $s(Q)=i$ to mean we chose the $i$th right-most vertex in the component and marked it with a star. See Figure~\ref{fig:markedexample}. For ease  we will write  $s(Q)$ or $s(L_k)$ for $s(D)$ assuming that we are looking at $Q=L_k$ in the larger scope of diagram $D$.

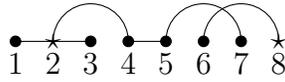
\begin{figure}
\begin{center}
\begin{tikzpicture}
\foreach \x in {0,2,3,...,6}
\filldraw[black] (\x/2,0) circle [radius=2pt];
\foreach \x in {1,7}
\filldraw[black] (\x/2,0) node {$\star$};
\foreach \x in {0,1,3}{
\draw[black] (\x/2,0)--(\x/2+.5,0);}
\foreach \x in {3,6,7}{
\(\arc{(\x/2,0)} \)}
\foreach \x in {}{
\(\tic{(\x/2,0)} \)}
\draw (0.0,-.3) node {1};
\draw (0.5,-.3) node {2};
\draw (1.0,-.3) node {3};
\draw (1.5,-.3) node {4};
\draw (2.0,-.3) node {5};
\draw (2.5,-.3) node {6};
\draw (3.0,-.3) node {7};
\draw (3.5,-.3) node {8};
\end{tikzpicture}
\end{center}
\caption{This diagram $D\in\cA'_L$ satisfies the IC-condition, has $s(D)=5$. }
\label{fig:markedexample}
\end{figure}

Now we are ready to describe the fixed-points. Fixed points are diagrams $D\in\cA'_L$ with no tic marks with  the following conditions. The diagram in Figure~\ref{fig:markedexample} satisfies all five conditions. 
\begin{enumerate}
\item[FP1. ] All sections satisfy the IC-condition except for the right-most ending at $n$ which can be a single  IC diagram.
\item[FP2. ] All IL diagrams $P$  of odd length ending before $n$ in the decomposition have $s(P)\neq 1$. 
\item[FP3. ] All IL diagrams $P$   followed by another IL diagram have $s(P)\neq 1$. 
\item[FP4. ] All IL diagrams $P$  of odd length in the decomposition have $s(P)\neq 2$. 
\item[FP5. ] If there is a $C_m\cdot L_{2k-1}$, $m,k\geq 1$, in the decomposition   then $s(L_{2k-1})\neq 3$ unless $C_m$ is immediately preceded by a $L_j$ with $s(L_j)=1$ or $C_m$ is one if the first two IC diagrams in the section. 
\end{enumerate}

To complete our bijection we will first split all  $D'\in\cA'_L$ into different cases. After we are sure we have all the cases then we will match up our cases and describe the rest of the involution. 

We have already considered the case where $D'=P\cdot Q$ where $P$ is not a fixed point and $Q$ is an IL or IC diagram of length at least one. Now we only have to consider cases about when $P$ is a fixed point. 
Our cases will be determined by the possibilities for $Q$  as well as the possibilities for $P_2$ where $P=P_1\cdot P_2$ and $P_2$ is an IC or IL diagram of length at least one if  $\ell(P)>0$. 

First we consider the cases where  $P_2$ satisfies one of the following three conditions. Either $P_2=C_m$ which fails the IC-condition because it is its own section,   $P_2=L_1$ with $s(P_2)\neq 1$ or  the length of $P_2$ is zero because the length of $P$ was zero. 

\begin{center}
\begin{tabular}{|p{1cm}|p{7cm}|p{8cm}|}
\hline
Case & Conditions on $P$ & Conditions on $Q$\\
\hline
\hline
\rtask{1a} \ref{1a}&
--$P_2=L_1$ with $s(P_2)\neq 1$ or

--$P_2=C_m$ that fails the IC-condition or 

--$\ell(P_2)=0$&
$Q=C_k$ with $t(Q)\geq 1$   \\
\hline
\rtask{1b} \ref{1b}&
--$P_2=L_1$ with $s(P_2)\neq 1$ or

--$P_2=C_m$ that fails the IC-condition or 

--$\ell(P_2)=0$&
$Q=L_k$\\
\hline
\rtask{fix1} \ref{fix1}&
--$P_2=L_1$ with $s(P_2)\neq 1$ or

--$P_2=C_m$ that fails the IC-condition or 

--$\ell(P_2)=0$
&
$Q=C_k$  with $t(Q)=0$\\
\hline
\end{tabular}
\end{center}

In all the remain cases $\ell(P_2)\geq 1$ and either $P_2$ is $C_m$ and the associated section satisfies the IC-condition, $P_2=L_1$ with $s(L_1)=1$ or $P_2=L_m$ with $m\geq 2$. Next we consider the cases where $P_2=L_m$  with $s(P_2)=1$.

\begin{center}
\begin{tabular}{|p{1cm}|p{7cm}|p{8cm}|}
\hline
Case & Conditions on $P$ & Conditions on $Q$\\
\hline
\hline
\rtask{3a} \ref{3a}&
$P_2=L_m$, $m\geq 1$, $s(P_2)=1$

&
$Q=C_k$, $t(Q)\geq 1$\\
\hline
\rtask{4a} \ref{4a}&
$P_2=L_{2m-1}$, $m\geq 1$, $s(P_2)=1$

&
$Q=C_k$, $t(Q)= 0$\\
\hline
\rtask{fix2} \ref{fix2}&
$P_2=L_{2m}$, $m\geq 1$, $s(P_2)=1$

&
$Q=C_k$, $t(Q)= 0$\\
\hline
\rtask{6a} \ref{6a}&
$P_2=L_{m}$, $m\geq 1$, $s(P_2)=1$

&
$Q=L_k$\\
\hline
\end{tabular}
\end{center}

Now we are left with the cases where the fixed-point $P$ ends in an IL diagram $L_m$, $m> 1$, with $s(P_2)\neq 1$ or $P_2$ is an IC diagram whose section satisfies the IC-condition. 
Since so many of the remaining  cases for $P_2$ fall under one of these cases we will say a diagram $P\in\cA'_L$ {\it satisfies the $*$-condition} if
\begin{itemize}
\item $P$ is a fixed-point of length at least 1, 
\item all sections in $P$  satisfy the IC-condition and
\item if $P$ ends in an IL diagram  then $s(P)\neq 1$ and the IL diagram has length at least two. 
\end{itemize}
 In all the remaining cases    $P$ satisfies the $*$-condition. This next section considers all possibilities where $Q$ is an IC diagram.

\begin{center}
\begin{tabular}{|p{1cm}|p{7cm}|p{8cm}|}
\hline
Case & Conditions on $P$ & Conditions on $Q$\\
\hline
\hline
\rtask{7b} \ref{7b}&
$P$ satisfies the $*$-condition   &
$Q=C_k$, $t(Q)=1$ splits $Q$ into $\alpha_1$ vertices before the tic mark and $\alpha_2$ afterwards,  $\alpha_1= \alpha_2$\\
\hline
\rtask{8b} \ref{8b}&
$P$ satisfies the $*$-condition &
$Q=C_k$, $t(Q)=1$ splits $Q$ into $\alpha_1$ vertices before the tic mark and $\alpha_2$ afterwards,  $\alpha_1> \alpha_2$\\
\hline
\rtask{4b} \ref{4b}&
$P$ satisfies the $*$-condition   &
$Q=C_k$, $t(Q)=1$ splits $Q$ into $\alpha_1$ vertices before the tic mark and $\alpha_2$ afterwards, $ \alpha_1<\alpha_2$\\
\hline
\rtask{3b} \ref{3b}&
$P$ satisfies the $*$-condition   &
$Q=C_k$, $t(Q)\geq 2$\\
\hline
\rtask{fix3} \ref{fix3}&
$P$ satisfies the $*$-condition   &
$Q=C_k$, $t(Q)=0$\\
\hline
\end{tabular}
\end{center}

This completes all cases where   $Q$ is an IC diagram. All that is left now is to consider the cases where $Q$ is an  IL diagram. We will be using a detailed condition to split up the remaining cases. We will say a fixed point $P$ satisfies the $**$-condition if it satisfies the $*$-condition and ends in a $C_m$, which is not one of the first two IC-diagrams in the section and in the case when $P$  ends in  $L_j\cdot C_m$ we have $s(L_j)\neq 1$.

\begin{center}
\begin{tabular}{|p{1cm}|p{7cm}|p{8cm}|}
\hline
Case & Conditions on $P$ & Conditions on $I$\\
\hline
\hline
\rtask{6b} \ref{6b}&
$P$ satisfies the $*$-condition  &
$Q=L_k$, $t(Q)\geq 1$\\
\hline
\rtask{7a} \ref{7a}&
$P$ satisfies the $*$-condition  &
$Q=L_{2k-1}$, $t(Q)=0$,
 $s(Q)=2$\\
\hline
\rtask{8a} \ref{8a}&
$P$ satisfies the $**$-condition   &
$Q=L_{2k-1}$, $t(Q)=0$,
 $s(Q)=3$\\
\hline
\rtask{fix4} \ref{fix4}&
$P$ satisfies the $*$-condition, but fails the $**$-condition
&
$Q=L_{2k-1}$, $t(Q)=0$,
 $s(Q)=3$\\
\hline
\rtask{fix5} \ref{fix5}&
$\tilde D$ satisfies the $*$-condition   &
$Q=L_{2k-1}$, $t(Q)=0$,
 $s(Q)\notin \{2,3\}$\\
\hline
\rtask{fix6} \ref{fix6}&
$ P$ satisfies the $*$-condition  &
$Q=L_{2k}$, $t(Q)=0$\\
\hline
\end{tabular}
\end{center}

We  are now ready to define the remaining part of our sign-reversing involution. We will do so by pairing up the cases above and then identifying the remaining cases as fixed points. Recall that we are considering diagrams $D'\in\cA'_L$ of the form $$D'=P\cdot Q=P_1\cdot P_2\cdot Q,$$
which has a permutation labeling, $P$ is a fixed point so has no tic marks, but $Q$ may or may not have tic marks.

In all these maps we will need to preserve the weight of our diagrams. We will need to have the size of the right-most piece connected to vertex $n$ to be the same size before and after the map, and also have the underlying integer partition be the same before and after the map. Additionally, we need the label of the right-most vertex to be the same before and after the map and if the right-most vertex is marked with a star we need its image to share that property. We will not outwardly address these four details since they will not be too hard to confirm, and we leave it for the reader. 

We instead focus on proving the involution is well defined, which is far more intricate since it revolves around very  precise conditions so needs careful consideration. The following lemma will  ease our  proof that the map is well defined.  

\begin{lemma}Consider $D$ in $\cA'_L$ without tic marks. 
\begin{enumerate}[(i)]
\item If $D=A\cdot B$ is a fixed point then $A$ is a fixed point. 
\item If $D=P \cdot L_m$ with $m\geq 1$ is a fixed-point then $P$ satisfies the $*$-condition. 
\item If $D=\bar D\cdot C_m$ with $m\geq 1$ has $\bar D$ satisfy the $*$-condition then so does $D$. 
\item If $D=P \cdot L_m$ where $P$ satisfies the $*$-condition and $s(L_m)=1$ then $P \cdot L_m$ is a fixed point. 
\item If  $P\cdot C_m$  satisfies the $**$-condition if and only if $P$ satisfies the $*$-condition.
\end{enumerate}
\label{lem:**}
\end{lemma}
\begin{proof}
Part (i) is true because if $D=A\cdot B$ is a fixed point, so satisfies all five conditions, then so does $A$, which also makes $A$ a fixed point.  

Say we have a fixed point $D=P \cdot L_m$ with $m\geq 1$. By part (i) we know $P$ is also a fixed point. By FP1 all but the right-most section of $D$ satisfy the IC-condition. If $m=1$ then $D$'s right-most section is a single vertex so the second right-most section, which is $P$'s right-most section, satisfies the IC-condition. If $m>1$ then  the right-most section of $D$ satisfies the IC-condition, which implies that the right-most section of  $P$ does as well. Thus, all sections in $P$ satisfy the IC-condition. Say that $P$ ends in an IL diagram $L_j$. By FP1 $j>1$ and by FP3 $s(L_j)\neq 1$ so $P$ must satisfy the $*$-condition. This proves part (ii). 

Say $D=P\cdot C_m$ with $m\geq 1$ where $P$ satisfies the $*$-condition. First, we will show that $D$ is a fixed point. We know $P$ satisfies the $*$-condition so all sections satisfy the IC-condition. Tacking on $C_m$ to the end does not change this, so FP1 is satisfied. Also, tacking on $C_m$  doesn't change the fact that FP3, FP4 and FP5 are satisfied. The only way that FP2 is not satisfied is when $P$ ends in $L_j$, $j\geq 1$. Note that because $P$ satisfied the $*$-condition we must have $s(L_j)\neq 1$, so FP2 is satisfied. Thus, $D$ is a fixed point. Because $P$ satisfies the $*$-condition all sections satisfy the IC-condition. Tacking on $C_m$ does not change this. Because $D$ ends in $C_m$ the last condition of the $*$-condition is satisfied and $D$ satisfies the $*$-condition. This completes part (iii).

Say $D=P \cdot L_m$ where $P$ satisfies the $*$-condition and $s(L_m)=1$. We want to show part (iv) by showing that $D$ is a fixed point. Because $P$ satisfies the $*$-condition all its sections satisfy the IC-condition. If $m=1$ then all sections of $D$, except the right-most which is a single vertex, satisfy the IC-condition. If $m>1$ then tacking on an $L_m$ does not change the fact that all sections satisfy the IC-condition, so $D$ satisfies FP1. 
Because $P$ is a fixed point  so satisfies FP2   the only way an $L_{2k-1}$ in the decomposition of $D$  may  have $s(L_{2k-1})=1$ is when it is  at the right end of $P$ if one exists there. However, in this case $s(L_{2k-1})\neq 1$ because  $P$ satisfies the $*$-condition so $D$ satisfies FP2. 
By a very similar argument $D$ also satisfies FP3. 
Now considering FP4.  Because $P$ is a fixed point the only way $D=P\cdot L_m$ can break FP4 is if $m$ is odd and $s(L_m)=2$, but we chose  $s(L_m)=1$. Hence, $D$ satisfies FP4. 
Because $P$ is a fixed point the only place FP5 could fail in $D$ is when $L_m$ is involved, meaning that $P$ ends in an $C_k$, $m$ is odd and $s(L_m)=3$. However, we set $s(L_m)=1$, so $D$ satisfies FP5. 
Hence, $D$ is a fixed point and we have finished part (iv). 

Say $P\cdot C_m$  satisfies the $**$-condition. By part (i) we know that $P$ is a fixed point. Because $P\cdot C_m$ satisfies the $**$-condition we know that $C_m$ is not one of the first two IC diagrams in its section. This implies that all sections of $P$ satisfy the IC-condition and if $P$ ends in  $L_j$ we know that $j\geq 2$ and further since $P\cdot C_m$  satisfies the $**$-condition that $s(L_j)\neq 1$. Hence, $P$ satisfies the $*$-condition. Now assume that $P$  satisfies the $*$-condition and we will consider $P\cdot C_m$. By part (iii) we know that $P\cdot C_m$ also satisfies the $*$-condition. Because all sections of $P$ satisfy the IC-condition we know that $C_m$ can't be one of the first two IC diagrams in its section. Finally consider the case where $C_m$ is preceded by $L_j$, which means that $P$ is ending in $L_j$. Because $P$ satisfies the $*$-condition we know that $s(L_j)\neq 1$, hence, $P\cdot C_m$ satisfies the $**$-condition completing part (v). 

\end{proof}

{\bf Involution part \case:} We will map Case~\ref{1a} with Case~\ref{1b}. In both of these cases we have either $\ell(P_2)=0$, $P_2=L_1$ with $s(P_2)\neq 1$ or $P_2=C_m$ which fails the IC-condition. In this latter case because we fail the IC-condition we can argue that either $P=C_m$ or $P=F \cdot L_1 \cdot C_m$ for 
$m\geq 1$. Further, we can say in  this latter case that because $P$ was a fixed point we know $s(F\cdot L_1)\neq 1$, 
and this matches the other situation where $P_2=L_1$. In summary, also including the situation where $\ell(P_2)=0$, the diagrams in these cases have forms $D'=C_m\cdot Q$ or $D'= F\cdot L_1 \cdot C_m\cdot Q$ where $m\geq 0$ and $s( F\cdot L_1)\neq 1$. We will focus on describing the map for $C_m\cdot Q$, $m\geq 0$, where $Q$ is either $C_k$  with $t(Q)\geq 1$ as in  Case~\ref{1a} or $L_k$ as in Case~\ref{1b}. 

First consider the case where $Q=C_k$, $k\geq 1$, which has at least one tic mark. The arc diagram $C_m\cdot C_k$ has pieces of sizes $\alpha_1+\alpha_2+\cdots +\alpha_{l}$ where $l\geq 2$ and $\alpha_1>m\geq 0$. Let us now determine an integer $j\in[l-1]$ as follows. This $j$ will indicate how we will split this integer composition into two compositions. If $l = 2$ let $j=1$. If there exists an $i\in [1,l-2]$ such that $\alpha_1+\cdots +\alpha_i\geq  \alpha_{i+2}+\cdots +\alpha_{l}$ then let $j$ be the smallest. If there is no such $i$ then let $j=l-1$. Using this $j$ we can define two segments of length 1 arcs. The first segment will have pieces of sizes $\alpha_{j+1}+\cdots +\alpha_{l}$ and the second segment will have pieces of sizes $\alpha_j+\cdots+ \alpha_1$. We label the pieces of each segment with their associated labels in the original diagram. We will interlace these two segments where the first one will end at the  right-most vertex and the second one will end at the second right-most vertex. There may be parts of one of these segments that sits outside and to the left of the interlacing, and in this case we replace these arcs with an IC diagram of the same length. The particular choice of $j$ assures us that none of these `outside arcs'  have a tic mark. The labels on all the pieces will remain the same, but since we created a new connected component that ends at the second right-most vertex we must choose a vertex in the piece associated to $\alpha_1$ to be marked with a star. Mark the $(\alpha_1-m)$th right-most vertex in this piece with a star. See Figure~\ref{fig:Part1} for an example. It is not too hard to see that the output is part of Case~\ref{1b} by the discussion in the last paragraph. Because of the star-marked vertex we can reverse this map. 

In more detail and to assure that this map is indeed invertible we will discuss how we map backwards from Case~\ref{1b} to Case~\ref{1a}. This time say $Q=L_k$, $k\geq 1$. The arc diagram $C_m\cdot L_k$ has two  connected components, one ending at the right-most vertex and one at the second right-most vertex with pieces of sizes $\alpha_{j+1}+\cdots +\alpha_{l}$ and  $\alpha_j+\cdots +\alpha_1$ respectively.  Because $C_m$ doesn't have a tic mark, this $j$ matches the choice of $j$ in the previous paragraph for the list of numbers $\alpha_1, \alpha_2,\ldots, \alpha_{l}$. 
We then map 
$$C_m\cdot L_k\mapsto C_{\alpha_1-s}\cdot C_{(s-1)+\alpha_2+\cdots+\alpha_{l}}$$
where $s=s(C_m\cdot L_k)\leq \alpha_1$, $C_{\alpha_1-s}$ has no tic marks, we remove the vertex star-marking on the piece associated to $\alpha_1$ and $C_{(s-1)+\alpha_2+\cdots+\alpha_{l}}$ has pieces $s+\alpha_2+\cdots+\alpha_{l}$, which has at least two pieces. It is not hard to see this undoes the map given in the paragraph above. See Figure~\ref{fig:Part1} for an example. 

\begin{figure}[h]
\begin{center}
\begin{tikzpicture}
\foreach \x in {0,...,8}
\filldraw[black] (\x/2,0) circle [radius=2pt];
\foreach \x in {9}
\filldraw[black] (\x/2,0) node {$\star$};
\foreach \x in {0,2}{
\draw[black] (\x/2,0)--(\x/2+.5,0);}
\foreach \x in {2,4,5,6,7,8,9}{
\(\arc{(\x/2,0)} \)}
\foreach \x in {4,6}{
\(\tic{(\x/2,0)} \)
\(\ticone{(2/2,0)} \)}
\draw (0.0,-.3) node {1};
\draw (0.5,-.3) node {2};
\draw (1.0,-.3) node {3};
\draw (1.5,-.3) node {4};
\draw (2.0,-.3) node {5};
\draw (2.5,-.3) node {6};
\draw (3.0,-.3) node {7};
\draw (3.5,-.3) node {8};
\draw (4.0,-.3) node {9};
\draw (4.5,-.3) node {10};
\end{tikzpicture}
\begin{tikzpicture}
\draw (-1,.2) node {$\longleftrightarrow$};
\foreach \x in {0,1,2,3,5,6,7,8}
\filldraw[black] (\x/2,0) circle [radius=2pt];
\foreach \x in {4,9}
\filldraw[black] (\x/2,0) node {$\star$};
\foreach \x in {0}{
\draw[black] (\x/2,0)--(\x/2+.5,0);}
\foreach \x in {3,...,9}{
\(\arc{(\x/2,0)} \)}
\foreach \x in {3,4}{
\(\tic{(\x/2,0)} \)}
\draw (0.0,-.3) node {5};
\draw (0.5,-.3) node {6};
\draw (1.0,-.3) node {4};
\draw (1.5,-.3) node {7};
\draw (2.0,-.3) node {1};
\draw (2.5,-.3) node {8};
\draw (3.0,-.3) node {2};
\draw (3.5,-.3) node {9};
\draw (4.0,-.3) node {3};
\draw (4.5,-.3) node {10};
\end{tikzpicture}
\end{center}
\caption{Example of the involution part 1.}
\label{fig:Part1}
\end{figure}
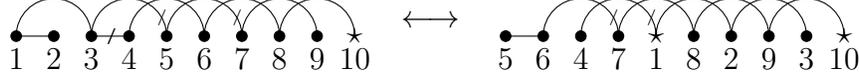

{\bf Involution part \case:} We will map diagrams from Case~\ref{3a} with diagrams from Case~\ref{3b}. Say $D'$ is a diagram from Case~\ref{3a}. Then $D'=P_1\cdot L_m\cdot C_k$ with $s(L_m)=1$,  $t(C_k)\geq 1$ and $P_1 \cdot L_m$ is a fixed point. Say that $C_k$ has pieces of sizes $\alpha_1+\cdots +\alpha_{l}$. 
We map 
$$P_1 \cdot L_m\cdot C_k\mapsto P_1 \cdot C_{m+k}$$
where we remove the star marking on the component connected to the second-most right vertex of $L_m$ and 
then we place tic marks on $C_{m+k}$ in a way that depends on whether $L_m$ has even or odd length. Say $L_m$ has even length, so has two components of sizes $j$ and $j+1$. Note that the component of size $j+1$ is connected to the left, which we want to preserve. In this case we set $C_{m+k}$ so it has pieces of sizes $(j+\alpha_1)+j+\alpha_2+\cdots+\alpha_{l}$. If instead $L_m$ has odd length, so has components of sizes $j$ and $j$, then we set $C_{m+k}$ so it has pieces of sizes $j+(j+\alpha_1-1)+\alpha_2+\cdots+\alpha_{l}$. In either case we keep the same labels on the same associated pieces. Also, in either case the output satisfies the conditions from Case~\ref{3b}, which are that $C_{m+k}$ has at least two tic marks and by Lemma~\ref{lem:**} (ii) we know that $P_1$ satisfies the $*$-condition. This is reversible because the even case corresponds to the first piece of $C_{m+k}$ being larger than the second. The odd case corresponds to  the first piece of $C_{m+k}$ being weakly smaller than the second. Further because  $P_1$ satisfies   the $*$-condition and because we set $s(L_m)=1$ in the backwards map by (iv) in Lemma~\ref{lem:**}  we have that $P_1\cdot L_m$ is a fixed point so we are well defined. See Figure~\ref{fig:Part2} for an example. 

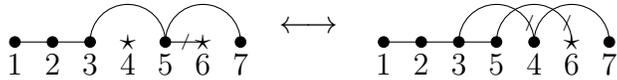
\begin{figure}[h]
\begin{center}
\begin{tikzpicture}
\foreach \x in {0,...,2,4,6}
\filldraw[black] (\x/2,0) circle [radius=2pt];
\foreach \x in {3,5}
\filldraw[black] (\x/2,0) node {$\star$};
\foreach \x in {0,1,4}{
\draw[black] (\x/2,0)--(\x/2+.5,0);}
\foreach \x in {4,6}{
\(\arc{(\x/2,0)} \)}
\foreach \x in {4}{
\(\ticone{(\x/2,0)} \)}
\draw (0.0,-.3) node {1};
\draw (0.5,-.3) node {2};
\draw (1.0,-.3) node {3};
\draw (1.5,-.3) node {4};
\draw (2.0,-.3) node {5};
\draw (2.5,-.3) node {6};
\draw (3.0,-.3) node {7};
\end{tikzpicture}
\begin{tikzpicture}
\draw (-1,.2) node {$\longleftrightarrow$};
\foreach \x in {0,...,4,6}
\filldraw[black] (\x/2,0) circle [radius=2pt];
\foreach \x in {5}
\filldraw[black] (\x/2,0) node {$\star$};
\foreach \x in {0,1,2}{
\draw[black] (\x/2,0)--(\x/2+.5,0);}
\foreach \x in {4,5,6}{
\(\arc{(\x/2,0)} \)}
\foreach \x in {4,5}{
\(\tic{(\x/2,0)} \)}
\draw (0.0,-.3) node {1};
\draw (0.5,-.3) node {2};
\draw (1.0,-.3) node {3};
\draw (1.5,-.3) node {5};
\draw (2.0,-.3) node {4};
\draw (2.5,-.3) node {6};
\draw (3.0,-.3) node {7};
\end{tikzpicture}
\end{center}
\caption{Example of the involution part 2.}
\label{fig:Part2}
\end{figure}

{\bf Involution part \case:} Here we will map Case~\ref{4a} to Case~\ref{4b}. Say $D'$ satisfies Case~\ref{4a}. Then $D'=P_1\cdot L_{2m-1}\cdot C_k$ where $m\geq 1$, $s(L_{2m-1})=1$, $t(C_k)=0$ and $P_1\cdot L_{2m-1}$ is a fixed point. We map 
$$P_1\cdot L_{2m-1}\cdot C_k\mapsto P_1\cdot C_{k+2m-1}$$
where we remove the star marking on the component connected to the vertex second from the right of $L_{2m-1}$,  $C_{k+2m-1}$ has two pieces of sizes $m+(k+m)$ and we keep the same labels on the same associated pieces. This satisfies Case~\ref{4b} because $1\leq m<k+m$ and $P_1$ satisfies the $*$-condition by Lemma~\ref{lem:**} (ii). Because of the condition on the two pieces in $C_{k+2m-1}$ this is reversible. Because $P_1$ satisfies the $*$-condition  and because we set $s(L_{2m-1})=1$ in the backwards map by (iv) in Lemma~\ref{lem:**}  we have that $P_1 \cdot L_{2m-1}$ is a fixed point so we are well defined. See Figure~\ref{fig:Part3} for an example.

\begin{figure}[h]
\begin{center}
\begin{tikzpicture}
\foreach \x in {0,2,3,4,...,6}
\filldraw[black] (\x/2,0) circle [radius=2pt];
\foreach \x in {1,7}
\filldraw[black] (\x/2,0) node {$\star$};
\foreach \x in {0,1,5}{
\draw[black] (\x/2,0)--(\x/2+.5,0);}
\foreach \x in {4,5,7}{
\(\arc{(\x/2,0)} \)}
\foreach \x in {}{
\(\tic{(\x/2,0)} \)}
\draw (0.0,-.3) node {1};
\draw (0.5,-.3) node {2};
\draw (1.0,-.3) node {3};
\draw (1.5,-.3) node {4};
\draw (2.0,-.3) node {5};
\draw (2.5,-.3) node {6};
\draw (3.0,-.3) node {7};
\draw (3.5,-.3) node {8};
\end{tikzpicture}
\begin{tikzpicture}
\draw (-1,.2) node {$\longleftrightarrow$};
\foreach \x in {0,1,2,3,4,...,6}
\filldraw[black] (\x/2,0) circle [radius=2pt];
\foreach \x in {7}
\filldraw[black] (\x/2,0) node {$\star$};
\foreach \x in {0,1,2}{
\draw[black] (\x/2,0)--(\x/2+.5,0);}
\foreach \x in {4,5,6,7}{
\(\arc{(\x/2,0)} \)}
\foreach \x in {4}{
\(\tic{(\x/2,0)} \)}
\draw (0.0,-.3) node {1};
\draw (0.5,-.3) node {2};
\draw (1.0,-.3) node {3};
\draw (1.5,-.3) node {5};
\draw (2.0,-.3) node {4};
\draw (2.5,-.3) node {6};
\draw (3.0,-.3) node {7};
\draw (3.5,-.3) node {8};
\end{tikzpicture}
\end{center}
\caption{Example of the involution part 3.}
\label{fig:Part3}
\end{figure}
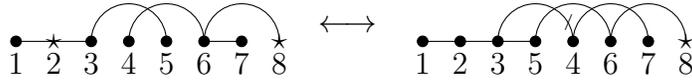

{\bf Involution part \case:} Here we will map  Case~\ref{6a} to Case~\ref{6b}. Say $D'$ is part of Case~\ref{6a} so $D'=P_1\cdot L_m\cdot L_k$ where $s(L_m)=1$ and $P_1 \cdot L_m$ is a fixed point. We map 
$$P_1\cdot L_m\cdot L_k\mapsto P_1\cdot L_{m+k},$$
which is the exact same diagram but we add a length 2 arc between $L_m$ and $L_k$ with a tic mark and remove the star-marking on the piece connected to the vertex second from the right in $L_m$. The output is certainly in Case~\ref{6b} since $L_{m+k}$ has at least one tic mark and $P_1$ satisfies the $*$-condition by Lemma~\ref{lem:**} (ii). This map is easily reversible by removing the arc with the tic mark we had just added, which is the first tic mark that appears from left-to-right vertex-wise in $L_{m+k}$. Because  $P_1$ satisfies the $*$-condition and because we set $s(L_m)=1$ in the backwards map by (iv) in Lemma~\ref{lem:**}  we have that $P_1 \cdot L_m$ is a fixed point so we are well defined. See Figure~\ref{fig:Part4} for an example.

\begin{figure}[h]
\begin{center}
\begin{tikzpicture}
\foreach \x in {0,...,2,4}
\filldraw[black] (\x/2,0) circle [radius=2pt];
\foreach \x in {3,5,6}
\filldraw[black] (\x/2,0) node {$\star$};
\foreach \x in {0,1}{
\draw[black] (\x/2,0)--(\x/2+.5,0);}
\foreach \x in {4,6}{
\(\arc{(\x/2,0)} \)}
\foreach \x in {}{
\(\ticone{(\x/2,0)} \)}
\draw (0.0,-.3) node {1};
\draw (0.5,-.3) node {2};
\draw (1.0,-.3) node {3};
\draw (1.5,-.3) node {4};
\draw (2.0,-.3) node {5};
\draw (2.5,-.3) node {6};
\draw (3.0,-.3) node {7};
\end{tikzpicture}
\begin{tikzpicture}
\draw (-1,.2) node {$\longleftrightarrow$};
\foreach \x in {0,...,4}
\filldraw[black] (\x/2,0) circle [radius=2pt];
\foreach \x in {5,6}
\filldraw[black] (\x/2,0) node {$\star$};
\foreach \x in {0,1}{
\draw[black] (\x/2,0)--(\x/2+.5,0);}
\foreach \x in {4,5,6}{
\(\arc{(\x/2,0)} \)}
\foreach \x in {5}{
\(\tic{(\x/2,0)} \)}
\draw (0.0,-.3) node {1};
\draw (0.5,-.3) node {2};
\draw (1.0,-.3) node {3};
\draw (1.5,-.3) node {5};
\draw (2.0,-.3) node {4};
\draw (2.5,-.3) node {6};
\draw (3.0,-.3) node {7};
\end{tikzpicture}
\end{center}
\caption{Example of the involution part 4.}
\label{fig:Part4}
\end{figure}
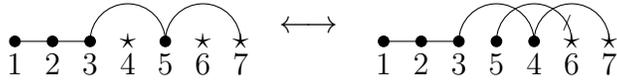
{\bf Involution part \case:} Here we will map  Case~\ref{7b} to Case~\ref{7a}. Say $D'$ satisfies Case~\ref{7b}. Then $D'=P\cdot C_k$, $k\geq 1$, where $P$ satisfies the $*$-condition and $t(C_k)=1$ with pieces of sizes $\alpha_1+\alpha_2$ with $\alpha_1= \alpha_2$. This means particularly that $k=2\alpha_1-1$ is odd. We map 
$$P\cdot C_{2\alpha_1-1}\mapsto P\cdot L_{2\alpha_1-1}$$
where there are no tic marks anywhere on $L_{2\alpha_1-1}$, we add a star-marking on the component attached to the vertex second from the right in $L_{2\alpha_1-1}$ so that $s(L_{2\alpha_1-1})=2$ and we keep all the vertex labeling the same in their respective pieces. 
The output is certainly in Case~\ref{7a} and this map is easily reversible and well defined. See Figure~\ref{fig:Part5} for an example.

\begin{figure}[h]
\begin{center}
\begin{tikzpicture}
\foreach \x in {0,1,2,3,5}
\filldraw[black] (\x/2,0) circle [radius=2pt];
\foreach \x in {4}
\filldraw[black] (\x/2,0) node {$\star$};
\foreach \x in {0,1,2}{
\draw[black] (\x/2,0)--(\x/2+.5,0);}
\foreach \x in {4,5}{
\(\arc{(\x/2,0)} \)}
\foreach \x in {4}{
\(\tic{(\x/2,0)} \)}
\draw (0.0,-.3) node {1};
\draw (0.5,-.3) node {2};
\draw (1.0,-.3) node {3};
\draw (1.5,-.3) node {4};
\draw (2.0,-.3) node {5};
\draw (2.5,-.3) node {6};
\end{tikzpicture}
\begin{tikzpicture}
\draw (-1,.2) node {$\longleftrightarrow$};
\foreach \x in {0,1,4,5}
\filldraw[black] (\x/2,0) circle [radius=2pt];
\foreach \x in {2,3}
\filldraw[black] (\x/2,0) node {$\star$};
\foreach \x in {0,1}{
\draw[black] (\x/2,0)--(\x/2+.5,0);}
\foreach \x in {4,5}{
\(\arc{(\x/2,0)} \)}
\foreach \x in {}{
\(\tic{(\x/2,0)} \)}
\draw (0.0,-.3) node {1};
\draw (0.5,-.3) node {2};
\draw (1.0,-.3) node {3};
\draw (1.5,-.3) node {5};
\draw (2.0,-.3) node {4};
\draw (2.5,-.3) node {6};
\end{tikzpicture}
\end{center}
\caption{Example of the involution part 5.}
\label{fig:Part5}
\end{figure}
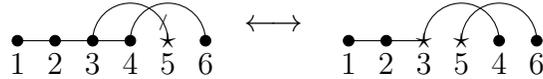

{\bf Involution part \case:} Here we will map  Case~\ref{8b} to Case~\ref{8a}. Say $D'$ satisfies Case~\ref{8b}. Then $D'=P\cdot C_k$, $k\geq 1$, where $P$ satisfies the $*$-condition and $t(C_k)=1$ with pieces of sizes $\alpha_1+\alpha_2$ with $\alpha_1> \alpha_2$. 
We map 
$$P\cdot C_{k}\mapsto P\cdot C_{\alpha_1-\alpha_2}\cdot L_{2\alpha_2-1}$$
where there are no tic marks anywhere, we add a star-marking on the component attached to the vertex second from the right in $L_{2\alpha_2-1}$ so that $s(L_{2\alpha_2-1})=3$ and we keep all the vertex labelings the same in their respective pieces. Because $\alpha_1-\alpha_2>0$ and Lemma~\ref{lem:**}  (v) we know that we must be in Case~\ref{8a}. The backwards map is that we replace $P\cdot C_{\alpha_1-\alpha_2}\cdot L_{2\alpha_2-1}$ with $P\cdot C_{\alpha_1+\alpha_2-1}$ so there is a tic mark breaking $C_{\alpha_1+\alpha_2-1}$ into two pieces of sizes $\alpha_1+\alpha_2$. By Lemma~\ref{lem:**}  (v) this is well defined. See Figure~\ref{fig:Part6} for an example.

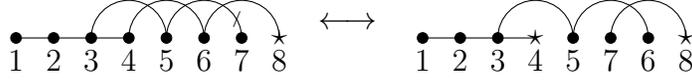
\begin{figure}[h]
\begin{center}
\begin{tikzpicture}
\foreach \x in {0,1,2,3,4,5,6}
\filldraw[black] (\x/2,0) circle [radius=2pt];
\foreach \x in {7}
\filldraw[black] (\x/2,0) node {$\star$};
\foreach \x in {0,1,2}{
\draw[black] (\x/2,0)--(\x/2+.5,0);}
\foreach \x in {4,5,6,7}{
\(\arc{(\x/2,0)} \)}
\foreach \x in {6}{
\(\tic{(\x/2,0)} \)}
\draw (0.0,-.3) node {1};
\draw (0.5,-.3) node {2};
\draw (1.0,-.3) node {3};
\draw (1.5,-.3) node {4};
\draw (2.0,-.3) node {5};
\draw (2.5,-.3) node {6};
\draw (3.0,-.3) node {7};
\draw (3.5,-.3) node {8};
\end{tikzpicture}
\begin{tikzpicture}
\draw (-1,.2) node {$\longleftrightarrow$};
\foreach \x in {0,1,2,4,5,6}
\filldraw[black] (\x/2,0) circle [radius=2pt];
\foreach \x in {3,7}
\filldraw[black] (\x/2,0) node {$\star$};
\foreach \x in {0,1,2}{
\draw[black] (\x/2,0)--(\x/2+.5,0);}
\foreach \x in {4,6,7}{
\(\arc{(\x/2,0)} \)}
\foreach \x in {}{
\(\tic{(\x/2,0)} \)}
\draw (0.0,-.3) node {1};
\draw (0.5,-.3) node {2};
\draw (1.0,-.3) node {3};
\draw (1.5,-.3) node {4};
\draw (2.0,-.3) node {5};
\draw (2.5,-.3) node {7};
\draw (3.0,-.3) node {6};
\draw (3.5,-.3) node {8};\end{tikzpicture}
\end{center}
\caption{Example of the involution part 6.}
\label{fig:Part6}
\end{figure}

{\bf Involution part \case:} The remaining Cases \ref{fix1},  \ref{fix2}, \ref{fix3}, \ref{fix4}, \ref{fix5} and \ref{fix6} 
are the fixed points of our map. Note that all diagrams from these six cases satisfy all five fixed point conditions, but we still need to show that all fixed points arise from one of these six cases. To be sure of this we will consider a fixed point $D=P\cdot Q$, where $Q$ is an IL or IC diagram of length at least one. By Lemma~\ref{lem:**}  (i) we know that $P$ is a fixed point. Let us consider the different cases for $P$ and $Q$. We will show each case falls under one of the six fixed point cases. 

First consider if $Q=C_k$. We could have $P$ ending in $L_{2m}$. If $s(L_{2m})=1$ we are in Case~\ref{fix2} and if $s(L_{2m})\neq 1$ then we are in Case~\ref{fix3}. If instead $P$ ends in $L_{2m-1}$, which must have $s(L_{2m-1})\neq 1$ because $D$ is a fixed point, then we are in Case~\ref{fix1} if $m=1$ and Case~\ref{fix3} otherwise. Lastly we may have $P$ ending in $C_m$, which falls under Case~\ref{fix1} if $C_m$ fails the IC-condition or Case~\ref{fix3} if $C_m$ satisfies the IC-condition. 

Next consider the case where $Q=L_k$ and in the case where $k$ is odd we have $s(L_k)\notin\{2,3\}$. We could have $P$ ending in $L_m$. Because $D$ is a fixed point we would have $m\geq 2$ and $s(L_m)\neq 1$. This falls under Case~\ref{fix5} if $k$ is odd and Case~\ref{fix6} if $k$ is even. We could also have $P$ ending in $C_m$ and because $D$ was a fixed point we know that $C_m$ is not the first IC diagram in the section. Similarly, this falls under Case~\ref{fix5} if $k$ is odd and Case~\ref{fix6} if $k$ is even. 

Lastly consider the case where $Q=L_k$ where $k$ is odd and $s(L_k)=3$. All these cases fall under Case~\ref{fix4}. We could have $P$ ending in $L_{2j}\cdot C_m$ where $s(L_{2j})=1$, $P$ ending in $C_m$ that is the second IC diagram in the section or $P$ could end in $L_m$ with $s(L_m)\neq 1$. 

\begin{thm}
For $n\geq 1$ the triangular ladder $TL_n$ is semi-symmetrized $e$-positive and so $e$-positive. 
\label{thm:TLpos}
\end{thm}

\begin{proof}
If we apply the sign reversing involution $\varphi$ on the signed set $\cA'_L$. We can see by the equation in Corollary~\ref{cor:TLsum} that $Y_{TL_n}$ is  semi-symmetrized $e$-positive and so $e$-positive because all fixed points have positive sign because they have no tic marks.
\end{proof}

Using the sign-reversing involution we defined in this section we can prove a few more graphs are semi-symmetrized $e$-positive.  We will discuss this in the next section, but we will need the following fact about the sign-reversing involution $\varphi$ we defined. 

\begin{prop}
If we are given a diagram $D'\in\cA_L$ such that $D'=C_{m-1}\cdot C_1\cdot  F'$ with $m\geq 2$ where $C_{m-1}\cdot C_1$ has no tic marks, then
\begin{enumerate}
\item  $\varphi(D')$ has the form $C_{m-1}\cdot C_1\cdot F''$  with no tic marks on $C_{m-1}\cdot C_1$, 
\item  the labels on the first $m$ vertices of $D'=C_{m-1}\cdot C_1\cdot  F'$ are the same as the labels on the first $m$ vertices of the output and 
\item if $D'$ doesn't have a star on the first vertex then neither does its output. 
\end{enumerate}
\label{prop:varphirestricts}
\end{prop}

\begin{proof}
Looking at all seven parts of the involution defined and equation~\eqref{eq:varphi}, where we also define the involution, we can clearly see that if $D'\in\cA_L$ such that $D'=C_{m-1}\cdot C_1\cdot  F'$ with $m\geq 2$ where $C_{m-1}\cdot C_1$ has no tic marks, then $\varphi(D')=C_{m-1}\cdot C_1\cdot F''$ also has no tic marks on $C_{m-1}\cdot C_1$, which proves part (i).  Similarly, we can see  that the labels on the first $m$ vertices of  $C_{m-1}\cdot C_1$ match those of the output, which gives up part (ii). 
 
Now we just have to mention why part (iii) is true, that if there is no star-markings on the first vertex of $D'=C_{m-1}\cdot C_1\cdot F'$ then  the image has the same property. If $D'=C_{m-1}\cdot C_1\cdot  F'$ is a fixed point or falls under the inductive case in equation~\eqref{eq:varphi} then we are done. 
The involution part 1 will not effect any part of $C_{m-1}\cdot C_1$. Parts 2, 3 and 4 of the involution at worst will place a star as far left as vertex $m+1$. Parts 5 and 6 of the involution will at worst  place a star as far left as vertex $m$. Because $m\geq 2$ we will never place a star on vertex one.  
\end{proof}

\section{More $e$-positive graphs}
\label{sec:More}
In this section we will use the involution $\varphi$ and other ideas from Section~\ref{sec:TL} along with results from Gebhard and Sagan~\cite{GS01}  to expand on the known families of $e$-positive graphs.  We will show   that any series of concatenations between triangular ladders and complete graphs results in an $e$-positive graph. This will expand on Gebhard and Sagan's result.

\begin{thm}[Gebhard and Sagan~\cite{GS01} Theorem 7.6 and 7.8]
If a graph $G$ is semi-symmetrized $e$-positive then so is $G\cdot K_m$ and $G\cdot TL_4$. 
\label{thm:gebsag_wedge}
\end{thm}

The involution $\varphi$ we defined in Section~\ref{sec:TL}  gives us exactly what we need to prove an expanded version of Theorem~\ref{thm:gebsag_wedge}. First, we need to discuss exactly how we can do this, so we present the following lemma that is essential to our proof. This lemma, when $\pi= 1$, is the result we have already proven, that $TL_n$ are semi-summarized $e$-positive.

\begin{lemma}
For any $e_{\pi}$, $\pi\vdash [m]$ with $m,n\geq 1$,  the sum
$$\sum_{D\in\cA(TL_n)}(-1)^{a(D)}e_{\pi}\uparrow_D$$
is semi-summarized $e$-positive. 
\label{lem:e_piInduce}
\end{lemma}
\begin{proof}
Because  the elementary basis is multiplicative and because of the relabeling proposition in Lemma~\ref{prop.relabeling} it suffices to show this formula for $\pi=[m]$. When $m=1$ this is exactly showing that the $TL_n$ are semi-summarized $e$-positive, which we have already done in Theorem~\ref{thm:TLpos}. 
Let $m\geq 2$. 
Note that by a  reason very similar to the proof of Theorem~\ref{thm:mainFormula}
$$\sum_{D\in\cA}(-1)^{a(D)}e_{[m]}\uparrow_D\equiv_{m+n}\frac{1}{(m+n)!}\sum_{D'\in \cA^{(m)}_L(TL_n)}(-1)^{t(D)}e_{\pi(D')}$$
where $\cA^{(m)}_L(TL_n)$ is the collection of  all arc diagrams in $\cA'_L(TL_{m+n})$ that start with $C_{m-2}\cdot C_1$, $t(C_{m-2}\cdot C_1)=0$ and the labeling on the first $m-1$ vertices don't have to be increasing but can be any permutation of the smallest $m-1$ labels in the first connected component.

We will define a slightly modified version of the sign-reversing involution on $\psi:\cA^{(m)}_L(TL_n)\rightarrow \cA^{(m)}_L(TL_n)$ by using our involution  $\varphi:\cA'_L(TL_{m+n})\rightarrow \cA'_L(TL_{m+n})$ and defining an invertible map 
$$h: \cA^{(m)}_L(TL_n)\rightarrow \{D'\in \cA^{(m+1)}_L(TL_n): \text{vertex 1 is labeled 1 with no star}\}.$$ 
We define $h$  by taking a  $D'\in \cA^{(m)}_L(TL_n)$  so $D'=C_{m-2}\cdot C_1\cdot F'$ with  define $D''=C_{m-1}\cdot C_1\cdot F'$ by letting $D''$ be $D'$ on vertices 2 through $m+n+1$ but place label 1 on vertex one,  increase all the labels by 1 on vertices 2 through $m+n+1$ and we also change the arcs on the first few vertices so that $D''$ starts with $C_{m-1}$ rather than $C_{m-2}$. This map $h(D')=D''$ is easily reversible, so $h$ is a bijection. We define 
$\psi:\cA^{(m)}_L(TL_n)\rightarrow \cA^{(m)}_L(TL_n)$ by mapping $D'\in \cA^{(m)}_L(TL_n)$ to $h^{-1}\circ\varphi\circ h(D')$. Note that we aren't using $\varphi$ exactly as we defined since our inputs $h(D')$ may have a permutation on the first $m$ vertices. However, by  Proposition~\ref{prop:varphirestricts} we know that $\varphi$ will not change the labeling on the first $m$ vertices of $h(D')$, so this is not an issue. This is the sign-reversing involution we need and by
 Proposition~\ref{prop:varphirestricts} it is well defined. 
\end{proof}

\begin{thm}
 If graph $G$ is semi-summarized $e$-positive then so is $G\cdot TL_m$.
 \label{thm:wedgeTL}
\end{thm}

\begin{proof}
First consider a graph $G$ that is is semi-summarized $e$-positive  and $G\cdot TL_m$. By a very similar reason as in the proof of Proposition~\ref{prop:Ginduce2} we have 
$$Y_{G\cdot TL_m}=\sum_{D\in\cA(TL_m)}(-1)^{a(D)}Y_G\uparrow_D.$$
Since  $G$ is a  semi-summarized $e$-positive graph we know that 
$$Y_G\equiv_n \sum_{(\pi)\subseteq \Pi_n}c_{(\pi)}e_{\pi}$$
where $c_{(\pi)}\geq 0$.  Combining the two equations above gives us
$$Y_{G\cdot TL_m}\equiv_{n+m-1}\sum_{(\pi)\subseteq \Pi_n}c_{(\pi)}\sum_{D\in\cA(TL_m)} (-1)^{a(D)}e_{(\pi)}\uparrow_D.$$
Using Lemma~\ref{lem:e_piInduce} we have that this is also semi-summarized $e$-positive. 
\end{proof}

Putting everything together we have an expanded version of Gebhard and Sagan's Theorem~\ref{thm:gebsag_wedge}. 
\begin{cor}
Any graph $G$ such that 
$$G=G_1\cdot G_2\cdots G_l$$
where $G_i=TL_{n_i}$ or $G_i=K_{n_i}$ is a semi-symmetrized $e$-positive graph, so is also an $e$-positive graph. 
\end{cor}

\begin{proof}
This follow immediately from Theorem~\ref{thm:gebsag_wedge} and Theorem~\ref{thm:wedgeTL}. 
\end{proof}

By computer calculation it can be confirmed that all unit interval graphs up to $7$ vertices are semi-symmetrized $e$-positive given the natural vertex labeling as given in the definition. It turns out that most permutations of this vertex labeling that do not fix $n$ or sent $n$ to $1$ give graphs that are not  semi-symmetrized $e$-positive. This means that labeling matters in  semi-symmetrized $e$-positivity, which is also the case for $e$-positivity in  quasi-symmetric chromatic symmetric functions defined by  Shareshian and Wachs~\cite{SW16}. It may be that  all unit interval graphs with the natural labeling are semi-symmetrized $e$-positive.

\section*{Acknowledgement}The author would like to thank Stephanie van Willigenburg for bringing attention to this open problem  and for her valuable comments. Also, thank you to Susanna Fishel and Stephanie van Willigenburg for helpful conversations.

\bibliographystyle{plain}

\end{document}